\documentclass[a4paper,10pt]{article}

\usepackage{amsmath,amsthm,amssymb,amsfonts}
\usepackage{geometry}
\usepackage{cleveref}
\usepackage{graphicx}
\usepackage{mathtools}
\usepackage[british]{babel}
\usepackage[shortlabels]{enumitem}

\newtheorem{thm}{Theorem}[section]
\newtheorem{cor}[thm]{Corollary}
\newtheorem{lem}[thm]{Lemma}
\newtheorem{prop}[thm]{Proposition}

\newtheorem{ques}[thm]{Question}

\theoremstyle{remark}
\newtheorem{rem}[thm]{Remark}
\newtheorem{con}[thm]{Convention}
\newtheorem{note}[thm]{Notation}

\theoremstyle{definition}
\newtheorem{defn}[thm]{Definition}
\newtheorem{exm}[thm]{Example}

\title{Tree approximation in quasi-trees}
\author{Alice Kerr}
\date{}

\begin{document}
	
\maketitle
	
\begin{abstract}
	 In this paper we investigate the geometric properties of quasi-trees, and prove some equivalent criteria. We give a general construction of a tree that approximates the ends of a geodesic space, and use this to prove that every quasi-tree is $(1,C)$-quasi-isometric to a simplicial tree. As a consequence, we show that Gromov's tree approximation lemma for hyperbolic spaces \cite{Gromov1987} can be improved in the case of quasi-trees to give a uniform approximation for any set of points, independent of cardinality. From this we show that having uniform tree approximation for finite subsets is equivalent to being able to uniformly approximate the entire space by a tree. As another consequence, we note that the boundary of a quasi-tree is isometric to the boundary of its approximating tree under a certain choice of visual metric, and that this gives a natural extension of the standard metric on the boundary of a tree.
\end{abstract}
	
\section{Introduction}

A \emph{quasi-tree} is a geodesic metric space that is quasi-isometric to a simplicial tree. Quasi-trees are hyperbolic spaces, however the existence of such a quasi-isometry means that they retain many more of the strong geometric properties that trees enjoy. For this reason, they can often act as a natural generalisation of trees, especially in the cases where the space simply being hyperbolic would be too weak of an assumption.

Quasi-trees are an object of interest not just because of their tree-like properties, but because of the group actions that they admit. Importantly, there are large classes of groups which have interesting actions on quasi-trees, but do not have any similar actions on trees. Many such examples were provided by Bestvina, Bromberg, and Fujiwara \cite{Bestvina2015}, including mapping class groups. It was additionally shown by Kim and Koberda that right-angled Artin groups admit natural acylindrical actions on quasi-trees, where the quasi-trees in question are constructed from the defining graph of the group \cite{Kim2013,Kim2014}. Balasubramanya later proved that, in fact, every acylindrically hyperbolic group admits an acylindrical action on some quasi-tree \cite{Balasubramanya2017}. This gives a strong motivation for proving new properties for quasi-trees, as they may allow us to obtain new results for these large classes of groups.

When attempting to find new properties of quasi-trees, we have two obvious options: generalise a result that is true for trees, or improve a result that has already been shown for hyperbolic spaces. The original aim for this paper was to do the latter for a result known as Gromov's tree approximation lemma, which essentially says that any finite set of points in a hyperbolic space can be approximated by a finite subset of a tree, with the error dependent on the number of points we are trying to approximate.

\begin{prop}[Gromov's tree approximation lemma]
	\emph{\cite[p.~155--157]{Gromov1987}}
	\label{IntroTreeApprox}
	Let $X$ be a $\delta$-hyperbolic geodesic metric space. Let $x_0,z_1,\ldots,z_n\in X$, and let $Y$ be a union of geodesic segments $\bigcup_{i=1}^n[x_0,z_i]$. Then there exist an $\mathbb{R}$-tree $T$ and a map $f:(Y,d)\to (T,d')$ such that:
	\begin{enumerate}[(1)]
		\item For all $1\leqslant i\leqslant n$, the restriction of $f$ to the geodesic segment $[x_0,z_i]$ is an isometry.
		\item For all $x,y\in Y$, we have that $d(x,y)-2\delta(\log_2(n)+1)\leqslant d'(f(x),f(y))\leqslant d(x,y)$.
	\end{enumerate}
\end{prop}

\begin{figure}[h]
	\centering
	\includegraphics[width=0.9\textwidth]{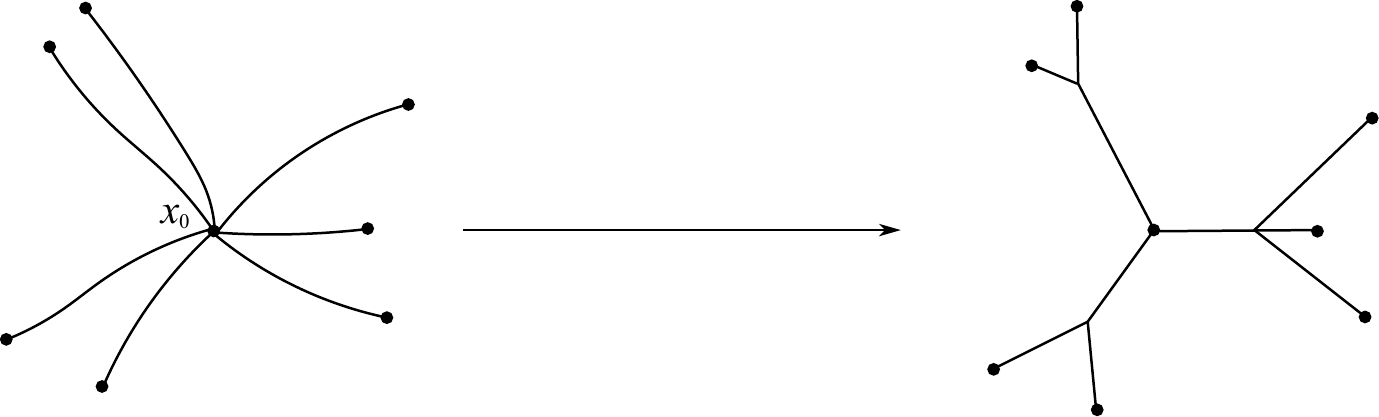}
	\caption{Tree approximation in a hyperbolic space}
\end{figure}

This result was used by Delzant and Steenbock \cite{Delzant2020} to generalise a theorem they proved about groups acting acylindrically on trees, from which they were able to obtain a theorem about groups acting acylindrically on hyperbolic spaces. The result for hyperbolic spaces was not as strong as their result for trees, which was a consequence of the logarithmic error that Gromov's tree approximation lemma introduced.

For general hyperbolic spaces, we would not expect to be able to improve the order of this error, however it is natural to ask whether we may be able to do better when we restrict ourselves to trying to approximate subsets of quasi-trees. In particular, we may ask if the approximation could be made to be uniform, in the sense that the error would no longer depend on the number of points being approximated.

It turns out that we can in fact prove a much stronger result than this, as we can show that the entire quasi-tree can be approximated by a tree with uniform error.

\begin{prop}[\Cref{QuasiIsometry2}]
	\label{IntroQuasiIsometry2}
	For every quasi-tree $X$, there exist an $\mathbb{R}$-tree $T$ and a constant $C\geqslant 0$ such that $X$ is $(1,C)$-quasi-isometric to $T$.
\end{prop}

One immediate implication of this, and the fact that $\mathbb{R}$-trees are $(1,C)$-quasi-isometric to simplicial trees (see \Cref{Simplicial}), is that quasi-trees are exactly the geodesic spaces that are $(1,C)$-quasi-isometric to simplicial trees. This is the central result of this paper.

\begin{thm}[\Cref{SimplicialIff}]
	\label{IntroSimplicialIff}
	A geodesic metric space $X$ is a quasi-tree if and only if it is $(1,C)$-quasi-isometric to a simplicial tree for some $C\geqslant 0$.
\end{thm}

A key step in proving these results is that we are able to obtain another equivalent definition of a quasi-tree, which can be compared to the definition of hyperbolic spaces using the Gromov product (see \Cref{HypDef1}).

\begin{prop}[\Cref{Constant}]
	\label{IntroConstant}
	A geodesic metric space $(X,d)$ is a quasi-tree if and only if there exists $A\geqslant 0$ such that any $x_0,x_1,\ldots,x_n\in X$ satisfy
	\begin{equation*}
		(x_1,x_n)_{x_0}\geqslant\min_{1\leqslant i\leqslant n-1}(x_i,x_{i+1})_{x_0}-A.
	\end{equation*}
\end{prop}

One application of \Cref{IntroSimplicialIff} is to quasi-actions on trees. It is known that a group has a cobounded quasi-action on a simplicial tree if and only if it has a quasi-conjugate isometric action on a quasi-tree \cite{Manning2006}. \Cref{IntroSimplicialIff} allows us to show that this implies the existence of a quasi-conjugate $(1,C)$-quasi-action on some simplicial tree.

\begin{prop}[\Cref{QuasiAct}]
	\label{IntroQuasiAct}
	If a finitely generated group admits a cobounded $(L,C)$-quasi-action on a simplicial tree for some $L\geqslant 1, C\geqslant 0$, then it admits a quasi-conjugate $(1,C')$-quasi-action on a simplicial tree for some $C'\geqslant 0$.
\end{prop}

We are also able to obtain results similar to \Cref{IntroSimplicialIff} for countable and locally finite simplicial trees.

\begin{prop}[\Cref{CountableFiniteCor2}]
	\label{IntroCountableFiniteCor2}
	We have the following:
	\begin{enumerate}[(1)]
		\item A geodesic metric space is $(L,C)$-quasi-isometric to a countable simplicial tree for some $L\geqslant 1, C\geqslant 0$ if and only if it is $(1,C')$-quasi-isometric to a countable simplicial tree for some $C'\geqslant 0$.
		\item A geodesic metric space is $(L,C)$-quasi-isometric to a locally finite simplicial tree for some $L\geqslant 1, C\geqslant 0$ if and only if it is $(1,C')$-quasi-isometric to a locally finite simplicial tree for some $C'\geqslant 0$.
	\end{enumerate}
\end{prop}

This allows us to get analogues to \Cref{IntroQuasiAct} in the case of quasi-actions on countable or locally finite trees.

\begin{cor}[\Cref{CountableFiniteQuasiAct}]
	\label{IntroCountableFiniteQuasiAct}
	We have the following:
	\begin{enumerate}[(1)]
		\item If a finitely generated group admits a cobounded $(L,C)$-quasi-action on a countable simplicial tree for some $L\geqslant 1, C\geqslant 0$, then it admits a quasi-conjugate $(1,C')$-quasi-action on a countable simplicial tree for some $C'\geqslant 0$.
		\item If a finitely generated group admits a cobounded $(L,C)$-quasi-action on a locally finite simplicial tree for some $L\geqslant 1, C\geqslant 0$, then it admits a quasi-conjugate $(1,C')$-quasi-action on a locally finite simplicial tree for some $C'\geqslant 0$.
	\end{enumerate}
\end{cor}

Returning to our original motivation, the tree and quasi-isometry in \Cref{IntroQuasiIsometry2} are constructed as an extension of the method used to prove Gromov's tree approximation lemma, and so they satisfy the same conditions. Restricting this quasi-isometry to subsets gives us our desired improvement of Gromov's result in the case of quasi-trees.

\begin{prop}[\Cref{Uniform}]
	\label{IntroUniform}
	Let $(X,d)$ be a quasi-tree. Let $x_0\in X$, and let $Z\subset X$. Let $Y$ be a union of geodesic segments $\bigcup_{z\in Z}[x_0,z]$. Then there exist an $\mathbb{R}$-tree $T$, a map $f:(Y,d)\to (T,d^*)$, and a constant $C\geqslant 0$ such that:
	\begin{enumerate}[(1)]
		\item For all $z\in Z$, the restriction of $f$ to the geodesic segment $[x_0,z]$ is an isometry.
		\item For all $x,y\in Y$, we have that $d(x,y)-C\leqslant d^*(f(x),f(y))\leqslant d(x,y)$.
	\end{enumerate}
\end{prop}

It is shown in \cite{Kerr2021a} that this improvement can be applied to Delzant and Steenbock's work \cite{Delzant2020} to produce a new result regarding product set growth in groups acting acylindrically on quasi-trees, with a particular application to right-angled Artin groups.

We note here that \Cref{IntroUniform} applies to subsets of any cardinality, rather than just finite subsets. On the other hand, it turns out that simply having uniform tree approximation for finite subsets is enough to ensure that the space in question is a quasi-tree, so quasi-trees are exactly the geodesic spaces for which uniform tree approximation is possible. In fact, even having uniform tree approximation for finite subsets is a stronger requirement than is necessary, as we still get a quasi-tree even if we relax the requirements on our tree approximation to allow for a multiplicative error.

\begin{prop}[\Cref{FiniteAprrox}]
	\label{IntroFiniteAprrox}
	Let $(X,d)$ be a geodesic metric space. The following are equivalent:
	\begin{enumerate}[(1)]
		\item $X$ is a quasi-tree.
		\item There exists a constant $C\geqslant 0$ such that for every finite subset $Z$ of $X$, there exist an $\mathbb{R}$-tree $(T,d^*)$ and a $(1,C)$-quasi-isometric embedding $f:(Z,d)\to (T,d^*)$.
		\item There exist constants $L\geqslant 1$ and $C\geqslant 0$ such that for every finite subset $Z$ of $X$, there exist an $\mathbb{R}$-tree $(T,d^*)$ and an $(L,C)$-quasi-isometric embedding $f:(Z,d)\to (T,d^*)$. 
	\end{enumerate}
\end{prop}

The remainder of the paper will mainly be concerned with the construction of the tree used in \Cref{IntroQuasiIsometry2}, and the ways in which its geometry relates to the original space. This tree can be constructed for any geodesic space, and can be thought of as collapsing the space along spheres around a basepoint, where two points will collapse together if they can be joined by a path that stays outside the central ball.

This construction can alternatively be visualised as collapsing the spaces along its ends, and we can show that the space of ends of any proper geodesic space will indeed be homeomorphic to the space of ends of the constructed tree (see \Cref{EndsHomeomorphic}). For this reason, we will often refer to such a tree as the \emph{end-approximating tree}. As a consequence of \Cref{EndsHomeomorphic}, we can also show that the space of ends of any proper geodesic space will in fact be homeomorphic to the space of ends of a locally finite simplicial tree (see \Cref{LocallyFiniteEnds}).

The boundary of the end-approximating tree can be compared to the boundary of the original space. For a hyperbolic space $X$, there exists a family of metrics on the boundary $\partial X$ known as \emph{visual metrics}, and given a tree $T$, and a choice of parameter, the visual metrics with that parameter on $\partial T$ are all bi-Lipschitz equivalent to a visual metric that can be written down explicitly in a standard way. It is known that a $(1,C)$-quasi-isometry between two hyperbolic geodesic spaces induces a bi-Lipschitz map between their boundaries \cite{Bonk2000}, and more specifically it is possible to choose a visual metric such that this map is an isometry. \Cref{IntroQuasiIsometry2} can therefore be used to show the following.

\begin{cor}[\Cref{SequentialRealv2}]
	\label{IntroSequentialRealv2}
	Let $X$ be a quasi-tree. There exist an $\mathbb{R}$-tree $T$ and a visual metric on $\partial X$ such that $\partial X$ endowed with this metric is isometric to $\partial T$ endowed with its standard visual metric, given a choice of basepoint and visual parameter.
\end{cor}

In particular, the visual metric that we choose on $\partial X$ can be defined directly from the geometry of $X$, and in the case where $X$ is an $\mathbb{R}$-tree it will be the standard visual metric on the boundary. This means that this metric can be viewed as an extension of the standard visual metric for the boundary of an $\mathbb{R}$-tree, with respect to a choice of basepoint and parameter.

\textbf{Structure of the paper:} In Section 2, we recall some basic facts about quasi-trees and hyperbolic geometry. In Section 3, we give the construction of our end-approximating tree for a geodesic metric space. In Section 4, we prove that the end-approximating tree for a quasi-tree is $(1,C)$-quasi-isometric to the original space, and then use this to show that quasi-trees are precisely those spaces that are $(1,C)$-quasi-isometric to some simplicial tree. We additionally consider the particular cases where the simplicial tree in question is countable, or locally finite. In Section 5, we prove the uniform version of Gromov's tree approximation lemma for quasi-trees. We also consider an alternative version of tree approximation, and show that this turns out to not be uniform for quasi-trees. In Section 6, we give an alternative description of the end-approximating tree, and use it to discuss the relationship between the ends of the tree and the ends of the space it is constructed from. In the case where this original space is a quasi-tree, we compare its boundary with the boundary of the tree.

\textbf{Acknowledgements:} The author would like to thank Cornelia Dru\c{t}u for asking whether tree approximation is uniform in quasi-trees, and for many other useful questions and suggestions along the way. The author is also grateful to Panos Papasoglou and David Hume for asking if a quasi-tree is $(1,C)$-quasi-isometric to a tree, and for providing feedback on an earlier version of this paper. Thank you also to the referee of this paper, and to the author's thesis examiners, Emmanuel Breuillard and Thomas Delzant, for several helpful comments and corrections.

\section{Preliminaries}

We first recall some basic definitions and lemmas. Those familiar with trees and quasi-trees may wish to go straight to Section 3, and use this section as a reference only.

\begin{note}
	Let $(X,d)$ be a metric space. Let $x_0\in X$ and $r\geqslant 0$. We will use $B(x_0,r)$ to denote the closed ball of radius $r$ around $x_0$, and $S(x_0,r)$ to denote the sphere of radius $r$ around $x_0$.
\end{note}

\begin{defn}
	Let $(X,d)$ be a metric space. Let $x_0,x,y\in X$. The \emph{Gromov product} of $x$ and $y$ at $x_0$ is
	\begin{equation*}
	(x,y)_{x_0}=\frac{1}{2}(d(x_0,x)+d(x_0,y)-d(x,y)).
	\end{equation*}
\end{defn}

Clearly, $(x,y)_{x_0}=(y,x)_{x_0}$, $(x_0,x)_{x_0}=0$, and $(x,x)_{x_0}=d(x_0,x)$. By the triangle inequality, $(x,y)_{x_0}\geqslant 0$.

\begin{defn}
	\label{HypDef1}
	Let $(X,d)$ be a metric space. Suppose there exists $\delta\geqslant 0$ such that for every $x_0,x,y,z\in X$, we have that
	\begin{equation*}
	(x,z)_{x_0}\geqslant\min\{(x,y)_{x_0},(y,z)_{x_0}\}-\delta
	\end{equation*}
	Then we say that $X$ is \emph{$\delta$-hyperbolic}. We say that $X$ is \emph{hyperbolic} if it is $\delta$-hyperbolic for some $\delta\geqslant 0$.
\end{defn}

The logarithm in the lower bound of Gromov's tree approximation lemma (\Cref{IntroTreeApprox}) comes from the following statement, which can be proved by induction directly from the above definition.

\begin{lem}
	\label{HyperbolicInequality}
	\emph{\cite[p.~155]{Gromov1987}}
	A metric space $X$ is $\delta$-hyperbolic in the sense of \Cref{HypDef1} if and only if any $x_0,x_1,\ldots,x_n\in X$ with $n\leqslant 2^k+1$ satisfy
	\begin{equation*}
	(x_1,x_n)_{x_0}\geqslant\min_{1\leqslant i\leqslant n-1}(x_i,x_{i+1})_{x_0}-k\delta.
	\end{equation*}
\end{lem}

In geodesic metric spaces, there is a commonly used equivalent definition that is more intuitive.

\begin{note}
	For $x,y$ in a geodesic metric space, we will use $[x,y]$ to represent any geodesic from $x$ to $y$.
\end{note}

\begin{defn}
	Let $(X,d)$ be a geodesic metric space. We say that $X$ is \emph{$\delta$-hyperbolic} for some $\delta\geqslant 0$ if for every $x,y,z\in X$, any choice of geodesic triangle $[x,y]\cup[y,z]\cup[z,x]$ is $\delta$-slim, meaning that if $p\in[x,y]$, then there exists $q\in[y,z]\cup[z,x]$ such that $d(p,q)\leqslant \delta$.
\end{defn}

\begin{rem}
	The $\delta$ in each definition is not generally the same. The exception is when a geodesic metric space is 0-hyperbolic, as in this case it will be 0-hyperbolic under both definitions.
\end{rem}

\begin{con}
	We will be mainly working in geodesic metric spaces, so we will take the second definition as standard when we say $\delta$-hyperbolic, unless otherwise specified.
\end{con}

In a geodesic hyperbolic space, the Gromov product of $x$ and $y$ at $x_0$ is approximately the distance between $x_0$ and $[x,y]$, as shown in the following standard lemmas.

\begin{lem}
	\label{Aux1}
	Let $(X,d)$ be a $\delta$-hyperbolic geodesic metric space, and let $x_0,x,y\in X$. Then there exists $z\in[x,y]$ such that $d(x_0,z)-2\delta\leqslant (x,y)_{x_0}$.
\end{lem}

\begin{lem}
	\label{Aux2}
	Let $(X,d)$ be a geodesic metric space, and let $x_0,x,y\in X$. Then $(x,y)_{x_0}\leqslant d(x_0,[x,y])$.
\end{lem}

\begin{defn}
	A metric space $(X,d)$ is an \emph{$\mathbb{R}$-tree} if for every $x,y\in X$, there exists a unique topological embedding $\alpha:[0,r]\to X$, and this embedding is a geodesic, so $d(x,y)=r$.
\end{defn}

The following equivalent characterisation is standard, see for example \cite{Bestvina2001}.

\begin{lem}
	A metric space is an $\mathbb{R}$-tree if and only if it is 0-hyperbolic and geodesic.
\end{lem}

\begin{rem}
	One nice consequence of these definitions is that in an $\mathbb{R}$-tree the Gromov product $(x,y)_{x_0}$ gives us exactly the distance between $x_0$ and the unique geodesic $[x,y]$. This is also the exact distance for which $[x_0,x]$ and $[x_0,y]$ coincide.
\end{rem}

\begin{defn}
	A \emph{simplicial tree} is a 1-dimensional simplicial complex that is an $\mathbb{R}$-tree under the path metric induced by considering each edge to be isometric to $[0,1]$.
\end{defn}

\begin{defn}
	Let $(X,d_X)$ and $(Y,d_Y)$ be metric spaces. A map $f:X\to Y$ is an $(L,C)$-\emph{quasi-isometry} if there exist constants $L\geqslant 1$, $C\geqslant 0$ such that:
	\begin{enumerate}[(1)]
		\item For every $x_1,x_2\in X$, we have that
		\begin{equation*}
		\frac{1}{L}d_X(x_1,x_2)-C\leqslant d_Y(f(x_1),f(x_2))\leqslant Ld_X(x_1,x_2)+C.
		\end{equation*}
		\item For every $y\in Y$, there exists $x\in X$ such that $d_Y(f(x),y)\leqslant C$.
	\end{enumerate}
	These two properties can be described as being coarsely bi-Lipschitz and coarsely surjective, respectively. If only the first property is satisfied, then the map is an $(L,C)$-\emph{quasi-isometric embedding}.
\end{defn}

\begin{rem}
	A $(1,C)$-quasi-isometry is sometimes known as a \emph{rough isometry}, amongst other terms.
\end{rem}

\begin{defn}
	Let $X$ and $Y$ be metric spaces. We say that $X$ is $(L,C)$-\emph{quasi-isometric} to $Y$ for some constants $L\geqslant 1,C\geqslant 0$ if there exists an $(L,C)$-quasi-isometry $f:X\to Y$. When the constants $L$ and $C$ are not specified, we will refer to $f$ simply as a quasi-isometry, and say that $X$ and $Y$ are quasi-isometric.
\end{defn}

\begin{rem}
	\label{QuasiInverse}
	If there exists an $(L,C)$-quasi-isometry $f:X\to Y$ for some $L\geqslant 1,C\geqslant 0$, then there exists an $(L',C')$-quasi-isometry $g:Y\to X$ for some $L'\geqslant 1,C'\geqslant 0$. In particular, if there exists a $(1,C)$-quasi-isometry $f:X\to Y$ for some $C\geqslant 0$, then there exists a $(1,3C)$-quasi-isometry $g:Y\to X$. When $L$ and $C$ are not given explicitly, we will therefore assume that they are large enough to hold in both directions. 
\end{rem}

Another fact which we will use about quasi-isometries is the following well-known lemma.

\begin{lem}
	\label{QuasiTransitive}
	Let $X$, $Y$, and $Z$ be metric spaces. Suppose there exist an $(1,C)$-quasi-isometry $f:X\to Y$, and a $(1,C')$-quasi-isometry $g:Y\to Z$, for some $C,C'\geqslant 0$. Then $g\circ f$ is a $(1,C+2C')$-quasi-isometry.
\end{lem}

\begin{defn}
	A \emph{quasi-tree} is a geodesic metric space that is quasi-isometric to a simplicial tree.
\end{defn}

Any quasi-tree will automatically be $\delta$-hyperbolic for some $\delta\geqslant 0$, as hyperbolicity is a quasi-isometry invariant. A highly useful alternative characterisation of quasi-trees was given by Manning.

\begin{thm}[Manning's bottleneck criterion]
	\label{BottleneckPrime}
	\emph{\cite{Manning2005}}
	A geodesic metric space $(X,d)$ is a quasi-tree if and only if there exists $\Delta\geqslant 0$ (the \emph{bottleneck constant}) such that for every  geodesic $[x,y]$ in $X$ with midpoint $m$, every path between $x$ and $y$ intersects the closed ball $B(m,\Delta)$.
\end{thm}

It is well known that this characterisation can be easily extended to the equivalent statement that every point on a geodesic in a quasi-tree has this property, not just the midpoint (see, for example, \cite{Bestvina2015}).

\begin{cor}
	\label{Bottleneck}
	A geodesic metric space $(X,d)$ is a quasi-tree if and only if there exists $\Delta\geqslant 0$ such that for every geodesic $[x,y]$ in $X$, and every $z\in[x,y]$, every path between $x$ and $y$ intersects the closed ball $B(z,\Delta)$.
\end{cor}

\begin{rem}
	\label{DeltaHyperbolic}
	It follows from \Cref{Bottleneck} that a quasi-tree with bottleneck constant $\Delta\geqslant 0$ is $\Delta$-hyperbolic.
\end{rem}

\section{End-approximating tree for a geodesic metric space}

In this section, we will give a construction of an $\mathbb{R}$-tree from a general geodesic metric space, in such a way that some of the original structure is retained. The idea is based on Gromov's proof of the tree approximation lemma in \cite{Gromov1987}. A similar explanation of such a construction can be found in \cite{Coornaert1990}, however we state the results here with greater generality.

\begin{defn}
	Let $(X,d)$ be a metric space. For $x,y\in X$, we will use the notation $S_{x,y}$ to mean the set of all finite sequences in $X$ between $x$ and $y$, so
	\begin{equation*}
	S_{x,y}=\{(x_1,\ldots ,x_n):x_1,\ldots ,x_n\in X,x_1=x,x_n=y,n\in\mathbb{N}\}.
	\end{equation*}
	Fix $x_0\in X$. We define $(x,y)'_{x_0}$ to be
	\begin{equation*}
	(x,y)'_{x_0}=\sup_{S_{x,y}}\min_{1\leqslant i\leqslant n-1}(x_i,x_{i+1})_{x_0}.
	\end{equation*}
	We finally define $d'$ on $X$ as
	\begin{equation*}
	d'(x,y)=d(x_0,x)+d(x_0,y)-2(x,y)'_{x_0}.
	\end{equation*}
\end{defn}

We begin with a couple of preliminary lemmas.

\begin{lem}
	\label{PMLem1}
	Let $(X,d)$ be a metric space, and let $x_0,x,y\in X$. Then $(x,y)'_{x_0}\leqslant d(x_0,x)$.
\end{lem}

\begin{proof}
	For $x,y\in X$, we can see that the triangle inequality for $d$ implies that $(x,y)_{x_0}\leqslant d(x_0,x)$. Hence for any sequence $(x_1,\ldots,x_n)$ in $S_{x,y}$, we have that
	\begin{equation*}
	\min_{1\leqslant i\leqslant n-1}(x_i,x_{i+1})_{x_0}\leqslant (x_1,x_2)_{x_0}=(x,x_2)_{x_0}\leqslant d(x_0,x),
	\end{equation*}
	so $(x,y)'_{x_0}\leqslant d(x_0,x)$.
\end{proof}

\begin{lem}
	\label{PMLem2}
	Let $(X,d)$ be a metric space, and let $x_0,x,y\in X$. Then we have that $(x,z)'_{x_0}\geqslant\min\{(x,y)'_{x_0},(y,z)'_{x_0}\}$.
\end{lem}

\begin{proof}
	This follows by concatenating sequences.
\end{proof}

We can use these lemmas to prove the following.

\begin{lem}
	\label{Pseudometric}
	The function $d':X\times X\to \mathbb{R}_{\geqslant 0}$ is a pseudometric on $X$, with the property that $d'\leqslant d$.
\end{lem}

\begin{proof}
	Fix $x_0\in X$. We first need to check that $d'$ actually maps to $\mathbb{R}_{\geqslant 0}$. Let $x,y\in X$. \Cref{PMLem1} tells us that $(x,y)'_{x_0}\leqslant d(x_0,x)$, and similarly that $(x,y)'_{x_0}\leqslant d(x_0,y)$. Hence we have that $d'(x,y)\geqslant 0$. The fact that $d'(x,y)<\infty$ is clear.
	
	We now want to show that $d'(x,x)=0$ for any $x\in X$. Note that $(x,y)$ is a valid sequence in $S_{x,y}$ for any $x,y\in X$, so $(x,y)'_{x_0}\geqslant(x,y)_{x_0}$, and therefore $d'(x,y)\leqslant d(x,y)$. Hence
	\begin{equation*}
		0\leqslant d'(x,x)\leqslant d(x,x)=0.
	\end{equation*}
	
	It is clear from the symmetry of the Gromov product for $d$ that $(x,y)'_{x_0}=(y,x)'_{x_0}$ for any $x,y\in X$, and hence clear that $d'(x,y)=d'(y,x)$.
	
	Finally, let $x,y,z\in X$. We can see that
	\begin{align*}
	d'(x,z)\leqslant d'(x,y)+d'(y,z) & \iff -2(x,z)'_{x_0}\leqslant 2d(x_0,y)-2((x,y)'_{x_0}+(y,z)'_{x_0})
	\\ & \iff 0\leqslant d(x_0,y)+(x,z)'_{x_0}-(x,y)'_{x_0}-(y,z)'_{x_0}.
	\end{align*}
	\Cref{PMLem2} tells us that
	\begin{equation*}
	(x,z)'_{x_0}\geqslant\min\{(x,y)'_{x_0},(y,z)'_{x_0}\},
	\end{equation*}
	so suppose without loss of generality that $(x,z)'_{x_0}\geqslant(x,y)'_{x_0}$. Then as $d(x_0,y)\geqslant(y,z)'_{x_0}$, we have shown that $d'$ satisfies the triangle inequality. Therefore $d'$ is a pseudometric on $X$.
\end{proof}

It is not hard to see that this is not a metric in general, as we could have $x\neq y$ but $d'(x,y)=0$. One example would be when $X=\mathbb{R}^n$ for some $n\geqslant 2$, and $x$ and $y$ lie on the same sphere around $x_0$ (see \Cref{Paths}).

\begin{defn}
	\label{TX}
	Let $(X,d)$ be a metric space, and fix $x_0\in X$. We define $(T_X,d^*)$ to be the metric space induced as a quotient of the pseudometric space $(X,d')$.

	Specifically, we consider the equivalence relation $\sim$, where $x\sim y$ if and only if $d'(x,y)=0$. We let $T_X=X/\sim$, and $d^*([x],[y])=d'(x,y)$. We let $([x],[y])^*_{[x_0]}$ be the Gromov product under the metric $d^*$.
\end{defn}

\begin{rem}
	We will consider the choice of basepoint to be fixed. Although the induced space may vary depending on this choice, it will not affect any of the properties that we consider here.
\end{rem}

\begin{lem}
	\label{ZeroHyp}
	Let $(X,d)$ be a metric space. The metric space $(T_X,d^*)$ is 0-hyperbolic.
\end{lem}

\begin{proof}
	We first note that for any $x\in X$, we have that $0\leqslant (x_0,x)'_{x_0}\leqslant d(x_0,x_0)=0$ by \Cref{PMLem1}, so in particular $d'(x_0,x)=d(x_0,x)$. We can now see that for any $x,y\in X$ we have that
	\begin{align*}
	(x,y)'_{x_0} & =\frac{1}{2}(d(x_0,x)+d(x_0,y)-d'(x,y))
	\\ & =\frac{1}{2}(d'(x_0,x)+d'(x_0,y)-d'(x,y))
	\\ & =\frac{1}{2}(d^*([x_0],[x])+d^*([x_0],[y])-d^*([x],[y]))
	\\ & =([x],[y])^*_{[x_0]},
	\end{align*}
	the Gromov product for the metric $d^*$. For any $x,y,z\in X$, we know from \Cref{PMLem2} that
	\begin{equation*}
	(x,z)'_{x_0}\geqslant\min\{(x,y)'_{x_0},(y,z)'_{x_0}\},
	\end{equation*}
	and therefore
	\begin{equation*}
	([x],[z])^*_{[x_0]}\geqslant\min\{([x],[y])^*_{[x_0]},([y],[z])^*_{[x_0]}\},
	\end{equation*}
	so $T_X$ is 0-hyperbolic.
\end{proof}

To get our desired $\mathbb{R}$-tree, we will need the additional assumption that the space $X$ is geodesic.

\begin{prop}
	\label{Tree}
	Let $(X,d)$ be a geodesic metric space. The metric space $(T_X,d^*)$ is an $\mathbb{R}$-tree.
\end{prop}

\begin{proof}
	We know that $T_X$ is 0-hyperbolic by \Cref{ZeroHyp}. It therefore only remains to show that $T_X$ is geodesic.
	
	Let $f:X\to T_X$ be the quotient map $f(x)=[x]$. As $d^*([x],[y])\leqslant d(x,y)$, it is obvious that $f$ is continuous, and moreover we can see that it is an isometry on any geodesic ray starting from $x_0$. Let $y\in X$, and let $x\in[x_0,y]$. We can note that
	\begin{equation*}
	(x,y)_{x_0}\leqslant(x,y)'_{x_0}\leqslant d(x_0,x)=(x,y)_{x_0},
	\end{equation*}
	so $(x,y)_{x_0}=(x,y)'_{x_0}$. Therefore
	\begin{equation*}
	d(x,y)=d'(x,y)=d^*([x],[y]).
	\end{equation*}
	In particular, the image of $[x,y]$ will be a geodesic between $[x]$ and $[y]$ in $T_X$.
	
	Now let $[x],[y]\in T_X$ be arbitrary. We want to construct a path in $T_X$ between $[x]$ and $[y]$ that has length $d^*([x],[y])$. Consider representatives $x$ and $y$ of these equivalence classes in $X$. We consider geodesics $[x_0,x]$ and $[x_0,y]$ in $X$, and pick the unique points $x'\in[x_0,x]$ and $y'\in[x_0,y]$ such that
	\begin{equation*}
	d(x_0,x')=d(x_0,y')=(x,y)'_{x_0}.
	\end{equation*}
	This means that
	\begin{align*}
	d(x,x')+d(y',y) & =d(x_0,x)+d(x_0,y)-2(x,y)'_{x_0}=d'(x,y),
	\end{align*}
	so $d^*([x],[x'])+d^*([y'],[y])=d^*([x],[y])$. Hence if we can show that $[x']=[y']$ then we can simply take the union of the geodesics $[[x],[x']]$ and $[[y'],[y]]$ to find a geodesic between $[x]$\linebreak and $[y]$.
	
	We have that
	\begin{equation*}
	d(x_0,x')=d(x_0,y')=([x],[y])^*_{[x_0]}=(x,y)'_{x_0}=\sup_{S_{x,y}}\min_{1\leqslant i\leqslant n-1}(x_i,x_{i+1})_{x_0}.
	\end{equation*}
	Recall that this supremum is over all finite sequences in $X$ beginning with $x$ and ending with $y$. Given that
	\begin{equation*}
	(x,x')_{x_0}=d(x_0,x')=d(x_0,y')=(y,y')_{x_0},
	\end{equation*}
	in this case we would get the same result if we took the supremum over all finite sequences beginning with $x,x',x$ and ending with $y,y',y$. This in turn is equivalent to taking the supremum over all finite sequences beginning with $x',x$ and ending with $y,y'$, which is no larger than taking the supremum over all sequences beginning with $x'$ and ending with $y'$, which is $S_{x',y'}$. Therefore
	\begin{equation*}
	d(x_0,x')=d(x_0,y')\leqslant (x',y')'_{x_0}\leqslant d(x_0,x')=d(x_0,y'),
	\end{equation*}
	so we have equality, and, in particular, $d'(x',y')=0$, so $[x']=[y']$. Hence $T_X$ is a geodesic 0-hyperbolic space, therefore it is an $\mathbb{R}$-tree.
\end{proof}

Informally, we construct $T_X$ from a geodesic metric space $X$ by collapsing $X$ along the spheres $S(x_0,r)$, where two points in $S(x_0,r)$ collapse to the same point in $T_X$ if and only if for every $\varepsilon>0$ there exists a path $\gamma$ between them in $X$ such that $d(x_0,\gamma)\geqslant r-\varepsilon$. Alternatively, we can view $T_X$ as the space that we get if we collapse $X$ along its ends.

This idea will be formalised in Section 6. It is not necessary for the proofs in Sections 4 or 5, however it may be a useful image to keep in mind, and it also motivates the language we use to describe $T_X$.

\begin{defn}
	Let $(X,d)$ be a geodesic metric space. We call $(T_X,d^*)$ the \emph{end-approximating tree} of $(X,d)$, and $f:X\to T_X$ defined by $f(x)=[x]$ the \emph{end-approximating map}.
\end{defn}

\begin{rem}
	\Cref{Pseudometric} tells us that the end-approximating map is non-expanding.
\end{rem}

We will often use slightly different terminology for the end-approximating tree of a quasi-tree.

\begin{con}
	When $(X,d)$ is a quasi-tree, we will occasionally simply refer to $(T_X,d^*)$ as an \emph{approximating tree}. This terminology is justified in Section 4.1.
\end{con}

\begin{figure}[h]
	\centering
	\includegraphics[width=0.7\textwidth]{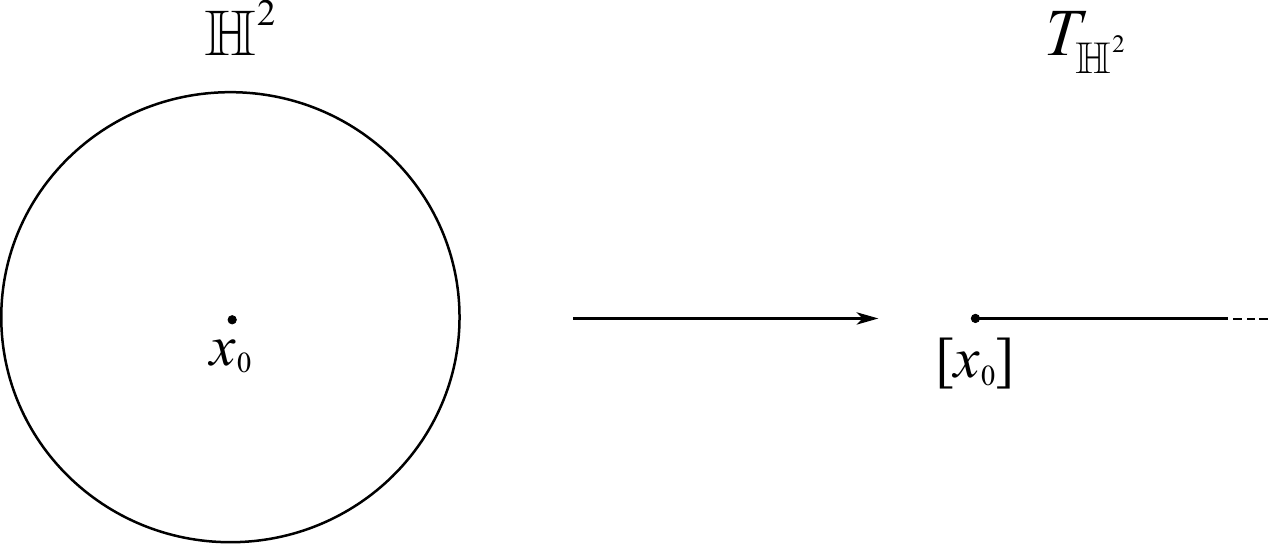}
	\caption{End-approximating tree of the hyperbolic plane}
\end{figure}

\begin{figure}[h]
	\centering
	\includegraphics[width=0.7\textwidth]{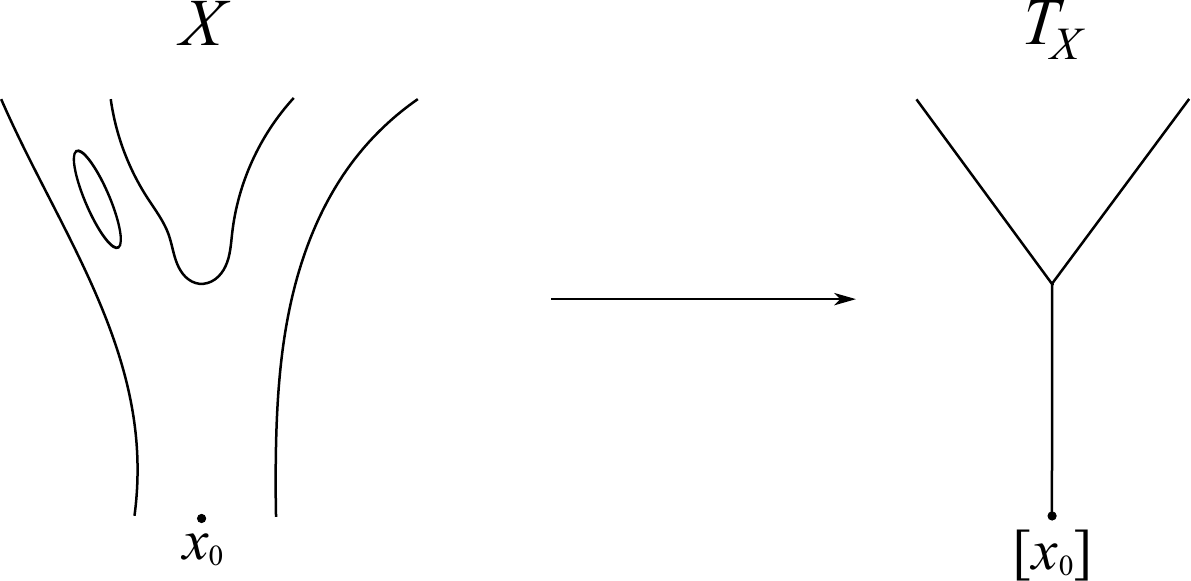}
	\caption{Approximating tree of a quasi-tree}
\end{figure}

\section{Quasi-isometries with trees}

For a general geodesic metric space $X$, the end-approximating tree $T_X$ may bear very little resemblance to $X$. For example, if $X=\mathbb{H}^n$ or $\mathbb{R}^n$ for some $n\geqslant 2$ then $T_X$ is isometric to $[0,\infty)$ (see Figure 2 and \Cref{OneEndedExample}). On the other hand, when $X$ is a quasi-tree, we might expect $T_X$ to look much more like the original space (see Figure 3). This makes sense intuitively, as in a quasi-tree we would expect there to be a limit on how far points can collapse together along spheres. We will show that in this case the end-approximating map $f:X\to T_X$ is in fact a $(1,C)$-quasi-isometry, and then use this to find a $(1,C')$-quasi-isometry from $X$ to a simplicial tree.

\subsection{$\mathbb{R}$-trees}

In this subsection, we will prove that if $X$ is a quasi-tree, then there exists some $C\geqslant 0$ such that the end-approximating map $f:X\to T_X$ is $(1,C)$-quasi-isometry. To achieve this, we must first prove a uniform version of \Cref{HyperbolicInequality} for quasi-trees, for which we will need Manning's bottleneck criterion.

\begin{prop}
	\label{Constant}
	A geodesic metric space $(X,d)$ is a quasi-tree if and only if there exists $A\geqslant 0$ such that any $x_0,x_1,\ldots,x_n\in X$ satisfy
	\begin{equation*}
	(x_1,x_n)_{x_0}\geqslant\min_{1\leqslant i\leqslant n-1}(x_i,x_{i+1})_{x_0}-A.
	\end{equation*}
\end{prop}

\begin{proof}
	Suppose $X$ is a $\delta$-hyperbolic quasi-tree with bottleneck constant $\Delta\geqslant 0$. Let $x_0,x_1,\ldots,x_n$ $\in X$ be arbitrary, and consider a geodesic $[x_1,x_n]$. By \Cref{Aux1} there exists $z\in[x_1,x_n]$ such that $d(x_0,z)-2\delta\leqslant (x_1,x_n)_{x_0}$.
	
	Consider the path $[x_1,x_2]\cup\cdots\cup[x_{n-1},x_n]$. By \Cref{Bottleneck}, the extended version of Manning's bottleneck criterion, for some $1\leqslant i\leqslant n-1$ there exists $z'\in[x_i,x_{i+1}]$ such that $z'\in B(z,\Delta)$, and $(x_i,x_{i+1})_{x_0}\leqslant d(x_0,z')$ by \Cref{Aux2}. We can put this together to get that
	\begin{align*}
	(x_i,x_{i+1})_{x_0}\leqslant d(x_0,z')\leqslant d(x_0,z)+d(z,z')\leqslant (x_1,x_n)_{x_0} +\Delta +2\delta.
	\end{align*}
	Hence $(x,y)_{x_0}\geqslant \min_{1\leqslant i\leqslant n-1}(x_i,x_{i+1})_{x_0}-(\Delta+2\delta)$. This concludes the proof of one direction.
	
	Now suppose that $(X,d)$ is a geodesic metric space, and that there exists $A\geqslant 0$ such that any $x_0,x_1,\ldots,x_n\in X$ satisfy $(x_1,x_n)_{x_0}\geqslant\min_{1\leqslant i\leqslant n-1}(x_i,x_{i+1})_{x_0}-A$. We want to show that $X$ satisfies Manning's bottleneck criterion.
	
	Let $x,y\in X$, with $[x,y]$ a geodesic between them with midpoint $m$, and let $\gamma$ be any path from $x$ to $y$. Choose $x=x_1,\ldots,x_n=y$ such that $x_i\in\gamma$ for all $1\leqslant i\leqslant n$ and $d(x_i,x_{i+1})\leqslant2A$ for all $1\leqslant i\leqslant n-1$.
	
	Note that $(x_1,x_n)_m=(x,y)_m=0$. We therefore have that
	\begin{align*}
	\min_{1\leqslant i\leqslant n-1}(x_i, x_{i+1})_m-A\leqslant (x_1,x_n)_m \implies \min_{1\leqslant i\leqslant n-1}(x_i,x_{i+1})_m\leqslant A,
	\end{align*}
	so for some $1\leqslant i\leqslant n-1$, we have that
	\begin{align*}
	(x_i,x_{i+1})_m\leqslant A & \implies \frac{1}{2}(d(m,x_i)+d(m,x_{i+1})-d(x_i,x_{i+1}))\leqslant A
	\\ & \implies d(m,x_i)+d(m,x_{i+1})\leqslant 4A.
	\end{align*}
	Hence there exists $1\leqslant i\leqslant n$ such that $d(m,x_i)\leqslant 2A$. As $x_i\in\gamma$, we have that there exists $z\in\gamma$ such that $d(m,z)\leqslant 2A$. By Manning's bottleneck criterion, $X$ is therefore a quasi-tree.
\end{proof}

We can now combine this with the construction of the end-approximating tree to show that for every quasi-tree the end-approximating map is a $(1,C)$-quasi-isometry.

\begin{prop}
	\label{QuasiIsometry}
	Let $(X,d)$ be a quasi-tree with bottleneck constant $\Delta\geqslant 0$ and hyperbolicity constant $\delta\geqslant 0$. Then the end-approximating map $f:X\to T_X$ is a $(1,2(\Delta+2\delta))$-quasi-isometry.
\end{prop}

\begin{proof}
	We first note that the end-approximating map is surjective, and hence is coarsely surjective. Now recall that for $[x],[y]\in T_X$, the distance between them is
	\begin{equation*}
	d^*([x],[y])=d'(x,y)=d(x_0,x)+d(x_0,y)-2(x,y)'_{x_0},
	\end{equation*}
	where
	\begin{equation*}
	(x,y)'_{x_0}=\sup_{S_{x,y}}\min_{1\leqslant i\leqslant n-1}(x_i,x_{i+1})_{x_0}.
	\end{equation*}
	We know that $(x,y)_{x_0}\leqslant(x,y)'_{x_0}$. By \Cref{Constant}, we also obtain that 
	\begin{equation*}
	(x,y)'_{x_0}\leqslant (x,y)_{x_0}+\Delta+2\delta.
	\end{equation*}
	Putting these inequalities together and substituting them into the definition of $d'(x,y)$ gives us that
	\begin{equation*}
	d(x,y)-2(\Delta+2\delta)\leqslant d'(x,y)\leqslant d(x,y),
	\end{equation*}
	and so
	\begin{equation*}
	d(x,y)-2(\Delta+2\delta)\leqslant d^*([x],[y])\leqslant d(x,y).
	\end{equation*}
	Hence the function $f$ is a $(1,2(\Delta+2\delta))$-quasi-isometry from $X$ to $T_X$.
\end{proof}

\begin{cor}
	\label{QuasiIsometry2}
	For every quasi-tree $X$, there exist an $\mathbb{R}$-tree $T$ and a constant $C\geqslant 0$ such that $X$ is $(1,C)$-quasi-isometric to $T$.
\end{cor}

\subsection{Simplicial trees}

We note here that $T_X$ will certainly not be a simplicial tree in general. This can be seen when $X$ is an $\mathbb{R}$-tree, since then $T_X=X$, so if $X$ is not simplicial then $T_X$ will not be either (see \Cref{XRTree}). Even in the relatively natural case that our quasi-tree $X$ is a graph and the basepoint $x_0\in X$ is a vertex, the $\mathbb{R}$-tree $T_X$ may fail to be simplicial with edge lengths 1. However, it is not hard to see that we can substitute our $\mathbb{R}$-tree for a simplicial tree with only an additive error in the quasi-isometry. The simplicial tree we construct here comes from a technique Manning uses in the proof of the bottleneck criterion \cite{Manning2005}.

\begin{figure}[h]
	\centering
	\includegraphics[width=0.7\textwidth]{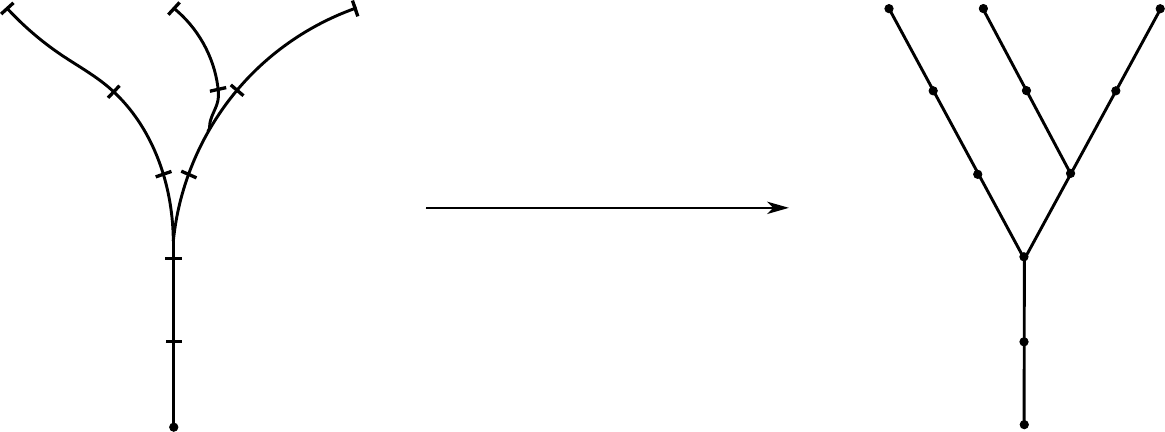}
	\caption{Constructing a simplicial tree from an $\mathbb{R}$-tree}
\end{figure}

Let $T$ be an $\mathbb{R}$-tree. We begin our construction by picking a basepoint $x_0\in T$, and letting $T_0=\{x_0\}$. We then let $\Gamma_0=(V_0,E_0)$ be a graph with $V_0=\{v_{x_0}\}$ and $E_0=\emptyset$.

Given $\Gamma_{k-1}$, we construct a new graph $\Gamma_k$ by adding the next layer of vertices. Let $T_k=S(x_0,k)$, and then let $U_k=\{v_x:x\in T_k\}$. We let the vertex set of $\Gamma_k$ be $V_k=V_{k-1}\cup U_k$, and the edge set be  $E_k=E_{k-1}\cup\{\{v_x,v_y\}: v_x\in U_{k-1}, v_y\in U_k, x\in[x_0,y]\}$.

If $\Gamma_{k-1}$ is a simplicial tree, then so is $\Gamma_k$. This is due to the fact that $T$ is an $\mathbb{R}$-tree, so for every $y\in T_k$ there is a unique $x\in T_{k-1}$ such that $x\in[x_0,y]$, and this rules out the existence of any cycles. We now let $\Gamma=(V,E)$, where $V=\bigcup_k V_k$ and $E=\bigcup_k E_k$. The graph $\Gamma$ is also a simplicial tree, as any cycle in $\Gamma$ would have to exist in some subgraph $\Gamma_k$.

We define a map $\psi:T\to\Gamma$ as follows. For every $x\in T$, let $x'\in T$ be the unique point such that $x'\in[x_0,x]$ and $d_T(x_0,x')=\lfloor d_T(x_0,x)\rfloor$. We then set $\psi(x)=v_{x'}$.

\begin{prop}
	\label{Simplicial}
	Every $\mathbb{R}$-tree is $(1,2)$-quasi-isometric to a simplicial tree.
\end{prop}

\begin{proof}
	Let $T$ be an $\mathbb{R}$-tree, and let $\Gamma$ be the simplicial tree constructed above. We will show that the map $\psi:T\to\Gamma$ is a $(1,2)$-quasi-isometry.
	
	We will first show that $\psi$ is coarsely surjective. Let $w\in \Gamma$. Then there exists a vertex $v\in \Gamma$ such that $d_\Gamma(w,v)\leqslant 1$. By the definition of $\Gamma$, there exists some $x\in T$ such that $v_x=v$. We therefore have that $\psi(x)=v$, so we are done.
	
	We will now show that $\psi$ is coarsely bi-Lipschitz. Let $x,y\in T$. Suppose first that $x\in[x_0,y]$. Then $\psi(x)\in[v_{x_0},\psi(y)]$, so $d_\Gamma(\psi(x),\psi(y))=\lfloor d_T(x_0,y)\rfloor-\lfloor d_T(x_0,x)\rfloor$. As $d_T(x,y)=d_T(x_0,y)-d_T(x_0,x)$, this means that
	\begin{equation*}
	d_T(x,y)-1\leqslant d_\Gamma(\psi(x),\psi(y))\leqslant d_T(x,y)+1.
	\end{equation*}
	
	Now suppose $x,y\in T$ are any two points. Let $z\in[x_0,x]\cap[x_0,y]$ be the unique point such that $d_T(x_0,z)=d_T(x_0,[x,y])$. Recall that $z'\in T$ is the unique point such that $z'\in[x_0,z]$ and $d_T(x_0,z')=\lfloor d_T(x_0,z)\rfloor$, so $[x_0,x]\cap[x_0,y]\cap S(x_0,d_T(x_0,z')+1)$ is empty. This means that $v_{z'}\in[\psi(x),\psi(y)]$, and so
	\begin{equation*}
	d_\Gamma(\psi(x),\psi(y))=d_\Gamma(\psi(x),\psi(z))+d_\Gamma(\psi(z),\psi(y))=d_T(x,z')+d_T(z',y).
	\end{equation*}
	Since
	\begin{align*}
	d_T(x,y)-2 & = d_T(x,z)+d_T(z,y)-2
	\\ & \leqslant d_T(x,z')+d_T(z',y)
	\\ & \leqslant d_T(x,z)+d_T(z,y)+2
	\\ & =d_T(x,y)+2,
	\end{align*}
	we can conclude that $\psi$ is a (1,2)-quasi-isometry from the $\mathbb{R}$-tree $T$ to the simplicial tree $\Gamma$.
\end{proof}

We noted in Section 2 that this means there is a (1,6)-quasi-isometry from $\Gamma$ to $T$, although it is in fact possible to improve this to show that $\Gamma$ is (1,2)-quasi-isometric to $T$. This can be proved via the map $\varphi:\Gamma\to T$, defined such that each edge $\{v_x,v_y\}$ is mapped isometrically to the geodesic $[x,y]$ in $T$. We provide a proof of this below, as $\varphi$ will be used in the proof of \Cref{CountableFiniteBackwards}.

\begin{prop}
	\label{SimplicialReverse}
	For every $\mathbb{R}$-tree $T$, there exists a simplicial tree that is $(1,2)$-quasi-isometric to $T$.
\end{prop}

\begin{proof}
	Let $\Gamma$ and $\varphi$ be as above. We will show that $\varphi$ is a $(1,2)$-quasi-isometry.
	
	We will first show that $\varphi$ is coarsely surjective. Let $x\in T$, and let $k=\lfloor d(x_0,x)\rfloor$. Then there exists some $y\in T_k$ such that $d_T(x,y)<1$, so as $y$ is the image of $v_y$ under $\varphi$ we therefore have that $\varphi$ is coarsely surjective.
	
	We will now show that $\varphi$ is coarsely bi-Lipschitz. Note that it is immediate from the definition of $\Gamma$ and $\varphi$ that for any $x\in \Gamma$, the geodesic $[v_{x_0},x]$ is mapped isometrically by $\varphi$ onto the geodesic $[x_0,\varphi(x)]$ in $T$.
	
	Let $x,y\in \Gamma$ be points that are not necessarily vertices, and consider the geodesic $[x,y]$. If $[x,y]\cap V=\emptyset$, then $[x,y]$ is contained in an edge of $\Gamma$, and therefore is mapped isometrically onto $T$ by $\varphi$. Otherwise, we can write $[x,y]$ as
	\begin{equation*}
	[x,y]=[x,v_1]\cup[v_1,v_2]\cup\cdots\cup[v_{n-1},v_n]\cup[v_n,y],
	\end{equation*}
	where every $v_i\in V\cap[x,y]$, and each subsegment has length greater than 0 but no greater than 1, so each $[v_i,v_{i+1}]$ is an edge in $\Gamma$.
	
	Suppose that $d_{\Gamma}(v_{x_0},[x,y])=d_{\Gamma}(v_{x_0},x)$, so $x$ is the nearest point to $v_{x_0}$. Then as $\Gamma$ is a tree, we have that $[x,y]\subset [v_{x_0},y]$, so $\varphi$ is an isometry on $[x,y]$. If $d_{\Gamma}(v_{x_0},[x,y])=d_{\Gamma}(v_{x_0},y)$, then, by the same reasoning, $\varphi$ is again an isometry on $[x,y]$.
	
	If on the other hand $d_{\Gamma}(v_{x_0},[x,y])=d_{\Gamma}(v_{x_0},v_i)$, then although $[x,v_i]$ and $[v_i,y]$ are mapped isometrically onto $[\varphi(x),\varphi(v_i)]$ and $[\varphi(v_i),\varphi(y)]$, we many have that the geodesics $[\varphi(v_{i-1}),\varphi(v_i)]$ and $[\varphi(v_i),\varphi(v_{i+1})]$ intersect at points other than $\varphi(v_i)$. Note that $\varphi(v_{i-1})\neq \varphi(v_{i+1})$ as otherwise we would have $v_{i-1}=v_{i+1}$, which would contradict our description of the geodesic $[x,y]$. This means that we can write $[\varphi(x),\varphi(y)]$ as
	\begin{equation*}
	[\varphi(x),\varphi(y)]=[\varphi(x),\varphi(v_{i-1})]\cup[\varphi(v_{i-1}),\varphi(v_{i+1})]\cup[\varphi(v_{i+1}),\varphi(y)].
	\end{equation*}
	Given we know that $[x,y]=[x,v_{i-1}]\cup[v_{i-1},v_{i+1}]\cup[v_{i+1},y]$, along with the fact that $d_{\Gamma}(x,v_{i-1})=d_T(\varphi(x),\varphi(v_{i-1}))$ and $d_{\Gamma}(v_{i+1},y)=d_T(\varphi(v_{i+1}),\varphi(y))$, it therefore only remains to compare $d_T(\varphi(v_{i-1}),\varphi(v_{i+1}))$ with $d_{\Gamma}(v_{i-1},v_{i+1})$.
	
	We can note that $d_T(\varphi(v_{i-1}),\varphi(v_{i+1}))>0$, and also that
	\begin{equation*}
	d_T(\varphi(v_{i-1}),\varphi(v_{i+1}))\leqslant d_T(\varphi(v_{i-1}),\varphi(v_i))+d_T(\varphi(v_i),\varphi(v_{i+1}))=2.
	\end{equation*}
	As $d_{\Gamma}(v_{i-1},v_{i+1})=2$, it immediately follows that
	\begin{equation*}
	d_{\Gamma}(x,y)-2\leqslant d_T(\varphi(x),\varphi(y))\leqslant d_{\Gamma}(x,y).
	\end{equation*}
	Therefore $\varphi$ is a $(1,2)$-quasi-isometry from the simplicial tree $\Gamma$ to the $\mathbb{R}$-tree $T$.
\end{proof}

We now combine \Cref{QuasiIsometry} and \Cref{Simplicial} to get the main result of this section.

\begin{thm}
	\label{SimplicialIff}
	A geodesic metric space $X$ is a quasi-tree if and only if it is $(1,C)$-quasi-isometric to a simplicial tree for some $C\geqslant 0$.
\end{thm}

\begin{proof}
	If a geodesic metric space $X$ is $(1,C)$-quasi-isometric to a simplicial tree, then it is a quasi-tree by definition. Now let $X$ be a quasi-tree. By \Cref{QuasiIsometry}, $X$ is $(1,C)$-quasi-isometric to the $\mathbb{R}$-tree $T_X$ for a constant $C=2(\Delta+2\delta)$, where $X$ is $\delta$-hyperbolic with bottleneck constant $\Delta$. By \Cref{Simplicial}, $T_X$ is $(1,2)$-quasi-isometric to some simplicial tree $\Gamma$,  so $X$ is $(1,C+4)$-quasi-isometric to $\Gamma$ by \Cref{QuasiTransitive}.
\end{proof}

\begin{rem}
	\label{AltReal}
	By \Cref{Simplicial}, we could equally define quasi-trees as being the geodesic metric spaces that are $(1,C)$-quasi-isometric to $\mathbb{R}$-trees.
\end{rem}

We can restate the above theorem in the following way.

\begin{cor}
	\label{QISimplicial}
	A geodesic metric space is $(L,C)$-quasi-isometric to a simplicial tree for some $L\geqslant 1, C\geqslant 0$ if and only if it is $(1,C')$-quasi-isometric to a simplicial tree for some $C'\geqslant 0$.
\end{cor}

\begin{rem}
	Note that these will not usually be the same simplicial tree. This is easily illustrated by the fact that there are simplicial trees that are $(L,C)$-quasi-isometric for some $L> 1, C\geqslant 0$, but not $(1,C')$-quasi-isometric for any $C'\geqslant 0$.
\end{rem}

\begin{rem}
	As with \Cref{AltReal}, we can easily apply \Cref{Simplicial} to substitute either or both of the simplicial trees in \Cref{QISimplicial} for an $\mathbb{R}$-tree. 
\end{rem}

The statement of \Cref{QISimplicial} was already known for quasi-isometries between graphs and $\mathbb{R}$, by a result of Manning.

\begin{lem}
	\emph{\cite{Manning2006}}
	\label{RealIff}
	A graph is $(L,C)$-quasi-isometric to $\mathbb{R}$ for some $L\geqslant 1, C\geqslant 0$ if and only if it is $(1,C')$-quasi-isometric to $\mathbb{R}$ for some $C'\geqslant 0$.
\end{lem}

One consequence of \Cref{SimplicialIff} is that it allows us to get rid of the requirement in Manning's result that the space in question is a graph.

\begin{cor}
	\label{RealIffv2}
	A geodesic metric space is $(L,C)$-quasi-isometric to $\mathbb{R}$ for some $L\geqslant 1, C\geqslant 0$ if and only if it is $(1,C')$-quasi-isometric to $\mathbb{R}$ for some $C'\geqslant 0$.
\end{cor}

\begin{proof}
	If a geodesic metric space $X$ is $(L,C)$-quasi-isometric to $\mathbb{R}$, then it is a quasi-tree, so by \Cref{SimplicialIff}, $X$ is $(1,C')$-quasi-isometric to a simplicial tree. This simplicial tree is quasi-isometric to $\mathbb{R}$, so by \Cref{RealIff}, it is $(1,C'')$-quasi-isometric to $\mathbb{R}$. Therefore $X$ is $(1,C'+2C'')$-quasi-isometric to $\mathbb{R}$ by \Cref{QuasiTransitive}.
\end{proof}

\subsection{Quasi-actions}

We note here that the improvement from \Cref{SimplicialIff} can also be applied to groups that quasi-act on simplicial trees.

\begin{defn}
	Let $(X,d)$ be a metric space, and let $G$ be a group. A map $A:G\times X\to X$ is an \emph{$(L,C)$-quasi-action} of $G$ on $X$ if there exist constants $L\geqslant 1$, $C\geqslant 0$ such that:
	\begin{enumerate}[(1)]
		\item For every $g\in G$, we have that the map $A(g,-):X\to X$ is an $(L,C)$-quasi-isometry.
		\item For every $g,h\in G$ and $x\in X$, we have that $d(A(g,A(h,x)),A(gh,x))\leqslant C$.
	\end{enumerate}
	A quasi-action is \emph{cobounded} if there exists a constant $C'\geqslant 0$ such that for every $x,y\in X$ there exists some $g\in G$ such that $d(A(g,x),y)\leqslant C'$.
\end{defn}

A group having a cobounded quasi-action on a simplicial tree is a trivial property, as we could consider the isometric action of any group on the tree which consists of a single vertex. To obtain a non-trivial statement, we therefore need a way of comparing quasi-actions.

\begin{defn}
	Let $(X,d_X)$ and $(Y,d_Y)$ be metric spaces, and let $A_X:G\times X\to X$ and $A_Y:G\times Y\to Y$ be quasi-actions. A map $f:X\to Y$ is \emph{coarsely equivariant} with respect to these quasi-actions if there exists $C\geqslant 0$ such that for every $x\in X$ and $g\in G$ we have that $d_Y(f(A_X(g,x)),A_Y(g,f(x)))\leqslant C$.
	
	If $f$ is also a quasi-isometry, then we call $f$ a \emph{quasi-conjugacy}. If there exists a quasi-conjugacy with respect to a pair of quasi-actions, then we say those quasi-actions are \emph{quasi-conjugate}.
\end{defn}

The following standard results tell us that quasi-actions being quasi-conjugate is an equivalence relation.

\begin{lem}
	\label{InverseQuasiConjugacy}
	Let $(X,d_X)$ and $(Y,d_Y)$ be metric spaces, let $A_X:G\times X\to X$ and $A_Y:G\times Y\to Y$ be quasi-actions, and let $f:X\to Y$ be a quasi-conjugacy with respect to these quasi-actions. Then the quasi-inverse given by \Cref{QuasiInverse} is also a quasi-conjugacy with respect to these quasi-actions.
\end{lem}

\begin{lem}
	\label{TranstitiveQuasiConjugacy}
	Let $(X,d_X)$, $(Y,d_Y)$, and $(Z,d_Z)$ be metric spaces, and let $A_X:G\times X\to X$, $A_Y:G\times Y\to Y$, and $A_Z:G\times Z\to Z$ be quasi-actions. Let $f:X\to Y$ be a quasi-conjugacy with respect to $A_X$ and $A_Y$, and let $h:Y\to Z$ be a quasi-conjugacy with respect to $A_Y$ and $A_Z$. Then $h\circ f:X\to Z$ is a quasi-conjugacy with respect to $A_X$ and $A_Z$.
\end{lem}

It is a result of Manning that every finitely generated group that has a cobounded quasi-action on a simplicial tree also admits a quasi-conjugate isometric action on a quasi-tree.

\begin{prop}
	\label{QuasiActImplies}
	\emph{\cite{Manning2006}}
	Let $G$ be a finitely generated group that has a cobounded quasi-action on a simplicial tree $\Gamma$. There exists a (possibly infinite) generating set $S$ for $G$ such that the Cayley graph $\text{Cay}(G,S)$ is quasi-isometric to $\Gamma$, and the isometric action of $G$ on $\text{Cay}(G,S)$ is quasi-conjugate to the quasi-action of $G$ on $\Gamma$.
\end{prop}

We therefore obtain the following.

\begin{prop}
	\label{QuasiAct}
	If a finitely generated group admits a cobounded $(L,C)$-quasi-action on a simplicial tree for some $L\geqslant 1, C\geqslant 0$, then it admits a quasi-conjugate $(1,C')$-quasi-action on a simplicial tree for some $C'\geqslant 0$.
\end{prop}

\begin{proof}
	Suppose that $G$ is a finitely generated group which has a cobounded quasi-action on a simplicial tree. Then by \Cref{QuasiActImplies}, there exists a generating set $S$ of $G$ such that $Cay(G,S)$ is a quasi-tree. By \Cref{SimplicialIff}, we have that $Cay(G,S)$ is $(1,C')$-quasi-isometric to some simplicial tree $\Gamma$. We want to show that the isometric action of $G$ on $Cay(G,S)$ induces a $(1,5C')$-quasi-action on $\Gamma$. Let $d$ be the metric on $Cay(G,S)$, and $d_{\Gamma}$ be the metric on $\Gamma$.
	
	Let $\varphi:Cay(G,S)\to\Gamma$ be a $(1,C')$-quasi-isometry, and let $\psi:\Gamma\to Cay(G,S)$ be the $(1,3C')$-quasi-isometry given by \Cref{QuasiInverse}. For every $g\in G$, we have that the map $\varphi\circ g\circ\psi:\Gamma\to\Gamma$ is a $(1,5C')$-quasi-isometry by \Cref{QuasiTransitive}.
	
	For $x\in\Gamma$ let $A(g,x)=\varphi\circ g\circ\psi(x)$. For every $g,h\in G$ and $x\in X$, we have that
	\begin{align*}
		d_{\Gamma}(A(g,A(h,x)),A(gh,x)) & = d_{\Gamma}(\varphi\circ g\circ\psi\circ\varphi\circ h\circ\psi(x),\varphi\circ g\circ h\circ\psi(x))
		\\ & \leqslant d(\psi\circ\varphi\circ h\circ\psi(x),h\circ\psi(x))+C'.
	\end{align*}
	By the definition of $\psi$ from \Cref{QuasiInverse}, we have that $d_{\Gamma}(\varphi(\psi(y)),y)\leqslant C'$ for every $y\in \Gamma$, so 
	\begin{equation*}
		d(\varphi\circ\psi\circ\varphi\circ h\circ\psi(x),\varphi\circ h\circ\psi(x))\leqslant C',
	\end{equation*}
	and therefore
	\begin{equation*}
		d(\psi\circ\varphi\circ h\circ\psi(x),h\circ\psi(x))\leqslant 2C'.
	\end{equation*}
	We conclude that 
	\begin{equation*}
		d_{\Gamma}(A(g,A(h,x)),A(gh,x))\leqslant 3C',
	\end{equation*}
	and so this is a $(1,5C')$-quasi-action.
	
	We now want to check that this quasi-action is quasi-conjugate to the isometric action of $G$ on $\text{Cay}(G,S)$. We already have that $\varphi:Cay(G,S)\to\Gamma$ is a quasi-isometry, and we can also see that
	\begin{align*}
		d_{\Gamma}(\varphi\circ g(x),A(g,\varphi(x))) & = d_{\Gamma}(\varphi\circ g(x),\varphi\circ g\circ\psi\circ\varphi(x))
		\\ & \leqslant d(x,\psi\circ\varphi(x))+C'
		\\ & \leqslant d_{\Gamma}(\varphi(x),\varphi\circ\psi\circ\varphi(x))+2C'
		\\ & \leqslant 3C'.
	\end{align*}
	Hence this is a quasi-conjugacy. \Cref{QuasiActImplies} tells us that the isometric action of $G$ on $\text{Cay}(G,S)$ is quasi-conjugate to the $(L,C)$-quasi-action on the original simplicial tree, so by \Cref{TranstitiveQuasiConjugacy}, we have that these are both quasi-conjugate to the $(1,5C')$-quasi-action of $G$ on $\Gamma$.
\end{proof}

As these actions are quasi-conjugate, it is well known that several properties of the original quasi-action will be inherited by the $(1,C')$-quasi-action. For instance, the $(1,C')$-quasi-action will also be cobounded. There is also a notion of individual group elements quasi-acting elliptically or loxodromically \cite{Manning2006}, which is also preserved under quasi-conjugacy.

\subsection{Countable and locally finite trees}

We note that the definition of a simplicial tree allows a vertex to have uncountably many neighbours. We here give one criteria for when an $\mathbb{R}$-tree is $(1,C)$-quasi-isometric to a countable or locally finite simplicial tree, which will allow us to obtain analogues of \Cref{QISimplicial} for countable and locally finite simplicial trees. We will also prove that every proper quasi-tree is $(1,C)$-quasi-isometric to a locally finite simplicial tree, which will be used at the end of Section 6.2.

\begin{prop}
	\label{CountableFiniteBackwards}
	Given an $\mathbb{R}$-tree $T$, we have the following:
	\begin{enumerate}[(1)]
		\item  If for some $r\geqslant 1$, there exists $R\geqslant0$ such that for every $x\in T$, the space $T\backslash B(x,r)$ has countably many connected components of diameter at least $R$, then $T$ is $(1,C)$-quasi-isometric to a countable simplicial tree.
		\item If for some $r\geqslant 1$, there exists $R\geqslant0$ such that for every $x\in T$, the space $T\backslash B(x,r)$ has finitely many connected components of diameter at least $R$, then $T$ is $(1,C)$-quasi-isometric to a locally finite simplicial tree.
	\end{enumerate}
\end{prop}

\begin{proof}
	We prove only the first part, as the proof of the second part is analogous. Suppose that $T$ is an $\mathbb{R}$-tree for which for some $r\geqslant 1$ there exists $R\geqslant0$ such that for every $x\in T$ the space $T\backslash B(x,r)$ has countably many connected components of diameter at least $R$. By \Cref{SimplicialReverse}, we know that there exists a simplicial tree $\Gamma$ with a $(1,2)$-quasi-isometry $\varphi:\Gamma\to T$. Here we take $\Gamma$ and $\varphi$ to be exactly the simplicial tree and quasi-isometry that we constructed in that proposition. Let $\Gamma'$ be obtained from $\Gamma$ by iteratively removing all leaves $n=\lceil r+R+4\rceil$ times. We then have that $\Gamma'$ is a simplicial tree that is $(1,2n)$-quasi-isometric to $\Gamma$, and so is $(1,2n+4)$-quasi-isometric to $T$ by \Cref{QuasiTransitive}. We want to show that $\Gamma'$ is countable.
	
	Let $v\in\Gamma'$ be a vertex. It suffices to show that its set of neighbours $S(v,1)$ in $\Gamma'$ is countable. Let $w\in S(v,1)$, then $w$ is also a neighbour of $v$ in $\Gamma$, and moreover there exists $z_w\in\Gamma$ such that $d_{\Gamma}(v,z_w)=n$ and $w\in[v,z_w]$.
	
	Recall that $\Gamma$ is constructed with respect to a basepoint $x_0\in T$, with $v_{x_0}$ being the corresponding vertex in $\Gamma$ such that $\varphi(v_{x_0})=x_0$. There is at most one $\widehat{w}\in S(v,1)$ such that $\widehat{w}\in[v_{x_0},v]$ in $\Gamma$, so $S(v,1)$ is countable if and only if $S(v,1)\backslash \{\widehat{w}\}$ is countable.
	
	Note that if $w,w'\in S(v,1)\backslash \{\widehat{w}\}$ and $w\neq w'$, then $z_w\neq z_{w'}$ as $\Gamma$ does not contain any cycles. We can also see that $d_{\Gamma}(z_w,z_{w'})=2n$, so $d_T(\varphi(z_w),\varphi(z_{w'}))\geqslant 2n-2$. As $v\in[v_{x_0},z_w]$, by the construction of $\varphi$ in \Cref{SimplicialReverse}, we get that $d_T(\varphi(v),\varphi(z_{w}))=n$. Similarly, $d_T(\varphi(v),\varphi(z_{w'}))=n$. We therefore get that
	\begin{align*}
	(\varphi(z_w),\varphi(z_{w'}))_{\varphi(v)} \leqslant \frac{1}{2}(2n-(2n-2))=1,
	\end{align*}
	so $\varphi(z_w)$ and $\varphi(z_{w'})$ lie in different connected components of $T\backslash B(\varphi(v),1)$. As $n\geqslant r$ and $r\geqslant 1$, we also have that $\varphi(z_w)$ and $\varphi(z_{w'})$ must lie in different connected components of $T\backslash B(\varphi(v),r)$. Note that each such connected component must have diameter at least $d_T(\varphi(v),\varphi(z_{w}))-1\geqslant n-3>R$. 
	
	By our initial assumption, there are at most countably many such connected components, so there are at most countably many $\varphi(z_w)$. As $w\neq w'$ implies that $\varphi(z_w)\neq \varphi(z_{w'})$, there are at most countably many $w\in S(v,1)\backslash \{\widehat{w}\}$, and so at most countably many $w\in S(v,1)$. Hence $\Gamma'$ is a countable simplicial tree, which is $(1,2n+4)$-quasi-isometric to $T$.
\end{proof}

\begin{rem}
	This is not true if we let $r=0$. As an example, consider the comb space $X$ formed by taking copies $X_n$ of $[0,\infty)$ indexed by $n\in\mathbb{Z}_{\geqslant 0}$, then attaching them to $[0,1]$ by identifying 0 in $X_0$ with 0 in $[0,1]$, and 0 in $X_n$ with $\frac{1}{n}$ in $[0,1]$ for $n\in\mathbb{N}$. This is an $\mathbb{R}$-tree, and for any $x\in X$, we can see that $X\backslash\{x\}$ has at most three connected components, however by \Cref{CountableFiniteForwards} below, we can see that that $X$ is not $(L,C)$-quasi-isometric to any locally finite simplicial tree.
\end{rem}

\begin{prop}
	\label{CountableFiniteForwards}
	Given an $\mathbb{R}$-tree $T$, we have the following:
	\begin{enumerate}[(1)]
		\item If $T$ is $(L,C)$-quasi-isometric to a countable simplicial tree, then for every $r\geqslant 0$, there exists $R\geqslant0$ such that for every $x\in T$, the space $T\backslash B(x,r)$ has countably many connected components of diameter at least $R$.
		\item If $T$ is $(L,C)$-quasi-isometric to a locally finite simplicial tree, then for every $r\geqslant 0$, there exists $R\geqslant0$ such that for every $x\in T$, the space $T\backslash B(x,r)$ has finitely many connected components of diameter at least $R$.
	\end{enumerate}
\end{prop}

\begin{proof}
	We again prove only the first part, as the proof of the second part is analogous. Let $T$ be an $\mathbb{R}$-tree, and suppose there exists a countable simplicial tree $\Gamma$ with a $(1,C)$-quasi-isometry $\varphi:\Gamma\to T$. Suppose also that there exists some $r\geqslant 0$ such that for every $R>0$, there exists $x\in T$ such that $T\backslash B(x,r)$ has uncountably many connected components of diameter at least $R$.
	
	Given $r\geqslant 0$, we let $R=3L+5C$, and pick such an $x\in T$. By coarse surjectivity, we know that there exist $v'\in\Gamma$ such that $d_T(\varphi(v'),x)\leqslant C$, and a vertex $v\in \Gamma$ such that $d_{\Gamma}(v,v')\leqslant 1$. Hence $d_T(\varphi(v),\varphi(v'))\leqslant L+C$, so $d_T(\varphi(v),x)\leqslant L+2C$.
	
	Let $\mathcal{C}_R$ be the set of connected components of $T\backslash B(x,r)$ of diameter greater than $R$, and for each $C\in\mathcal{C}_R$ pick $x_C\in C$ such that $r+\frac{1}{2}R\leqslant d_T(x,x_C)\leqslant r+R$, which must always exist as $T$ is an $\mathbb{R}$-tree. Let $v_C\in\Gamma$ be a vertex such that $d_T(\varphi(v_C),x_C)\leqslant L+2C$, then $d_T(\varphi(v),\varphi(v_C))\leqslant r+R+2(L+2C)$. Therefore $d_{\Gamma}(v,v_C)\leqslant L(R+r+2L+5C)$.
	
	Note that for $C\neq C'$, we have that
	\begin{align*}
	d_T(x_C,x_{C'})\geqslant R & \implies d_T(\varphi(v_C),\varphi(v_{C'}))\geqslant R-2(L+2C)
	\\ & \implies d_{\Gamma}(v_C,v_{C'})\geqslant \frac{1}{L}(R-2L-5C)>0.
	\end{align*} 
	This means that $v_C\neq v_{C'}$, so there are uncountably many such $v_C$'s. Recall that when $m=L(R+r+2L+5C)$, we have that $d_{\Gamma}(v,v_C)\leqslant m$ for every $C\in\mathcal{C}_R$, so this means that there are uncountably many vertices in $B(v,m)$. This contradicts $\Gamma$ being countable. Hence if $T$ is $(1,C)$-quasi-isometric to a countable simplicial tree, we have that for every $r\geqslant 0$, there exists $R>0$ such that for every $x\in T$, the space $T\backslash B(x,r)$ has countably many connected components of diameter at least $R$.
\end{proof}

We therefore obtain some equivalent criteria for when an $\mathbb{R}$-tree is quasi-isometric to a countable or locally finite simplicial tree, including a partial analogue to \Cref{QISimplicial}.

\begin{cor}
	\label{CountableFinite1}
	Given an $\mathbb{R}$-tree $T$, the following are equivalent:
	\begin{enumerate}[(1)]
		\item $T$ is $(L,C)$-quasi-isometric to a countable simplicial tree for some $L\geqslant 1, C\geqslant 0$.
		\item $T$ is $(1,C')$-quasi-isometric to a countable simplicial tree for some $C'\geqslant 0$.
		\item For some $r\geqslant 1$, there exists $R>0$ such that for every $x\in T$, the space $T\backslash B(x,r)$ has countably many connected components of diameter at least $R$.
		\item For every $r\geqslant 0$, there exists $R\geqslant0$ such that for every $x\in T$, the space $T\backslash B(x,r)$ has countably many connected components of diameter at least $R$.
	\end{enumerate}
\end{cor}

\begin{cor}
	\label{CountableFinite2}
	Given an $\mathbb{R}$-tree $T$, the following are equivalent:
	\begin{enumerate}[(1)]
		\item $T$ is $(L,C)$-quasi-isometric to a locally finite simplicial tree for some $L\geqslant 1, C\geqslant 0$.
		\item $T$ is $(1,C')$-quasi-isometric to a locally finite simplicial tree for some $C'\geqslant 0$.
		\item For some $r\geqslant 1$, there exists $R>0$ such that for every $x\in T$, the space $T\backslash B(x,r)$ has finitely many connected components of diameter at least $R$.
		\item For every $r\geqslant 0$, there exists $R\geqslant0$ such that for every $x\in T$, the space $T\backslash B(x,r)$ has finitely many connected components of diameter at least $R$.
	\end{enumerate}
\end{cor}

We can then extend the first equivalences to geodesic metric spaces to get a full analogue of \Cref{QISimplicial}.

\begin{cor}
	\label{CountableFiniteCor2}
	We have the following:
	\begin{enumerate}[(1)]
		\item A geodesic metric space is $(L,C)$-quasi-isometric to a countable simplicial tree for some $L\geqslant 1, C\geqslant 0$ if and only if it is $(1,C')$-quasi-isometric to a countable simplicial tree for some $C'\geqslant 0$.
		\item A geodesic metric space is $(L,C)$-quasi-isometric to a locally finite simplicial tree for some $L\geqslant 1, C\geqslant 0$ if and only if it is $(1,C')$-quasi-isometric to a locally finite simplicial tree for some $C'\geqslant 0$.
	\end{enumerate}
\end{cor}

\begin{proof}
	Let $X$ be a geodesic metric space that is quasi-isometric to a countable simplicial tree. We then have that $X$ is a quasi-tree, so there exists $C'\geqslant 0$ such that $X$ is $(1,C')$-quasi-isometric to $T_X$. This means that  $T_X$ is an $\mathbb{R}$-tree that is quasi-isometric to a countable simplicial tree, so by \Cref{CountableFinite1}, there exists $C''\geqslant 0$ such that $T_X$ is $(1,C'')$-quasi-isometric to a countable simplicial tree. We conclude that $X$ is $(1,C'+2C'')$-quasi-isometric to a countable simplicial tree by \Cref{QuasiTransitive}. The reverse direction is obvious, and the second part follows using the same method.
\end{proof}

This allows us to get analogues to \Cref{QuasiAct} in the case of quasi-actions on countable or locally finite simplicial trees, using the same reasoning as in the proof of \Cref{QuasiAct}.

\begin{cor}
	\label{CountableFiniteQuasiAct}
	We have the following:
	\begin{enumerate}[(1)]
		\item If a finitely generated group admits a cobounded $(L,C)$-quasi-action on a countable simplicial tree for some $L\geqslant 1, C\geqslant 0$, then it admits a quasi-conjugate $(1,C')$-quasi-action on a countable simplicial tree for some $C'\geqslant 0$.
		\item If a finitely generated group admits a cobounded $(L,C)$-quasi-action on a locally finite simplicial tree for some $L\geqslant 1, C\geqslant 0$, then it admits a quasi-conjugate $(1,C')$-quasi-action on a locally finite simplicial tree for some $C'\geqslant 0$.
	\end{enumerate}
\end{cor}

We conclude this section by noting that if our quasi-tree is a proper space, then we can use \Cref{CountableFiniteBackwards} to obtain a $(1,C)$-quasi-isometry with a locally finite simplicial tree.

\begin{cor}
	\label{ProperFinite}
	Every proper quasi-tree is $(1,C)$-quasi-isometric to a locally finite simplicial tree.
\end{cor}

\begin{proof}
	We first show that every proper $\mathbb{R}$-tree has this property. Let $T$ be an $\mathbb{R}$-tree, and suppose it is not $(1,C)$-quasi-isometric to a locally finite simplicial tree. By \Cref{CountableFiniteBackwards}, we can find $x\in T$ such that $T\backslash B(x,1)$ has infinitely many connected components of diameter greater than 2. Let $\mathcal{C}$ be the set of these connected components.
	
	For every $C\in\mathcal{C}$, we can pick $x_C\in C$ such that $d(x,x_C)=2$. Note that, as $T$ is an $\mathbb{R}$-tree, for every $C,C'\in\mathcal{C}$ such that $C\neq C'$, we have that $d(x_C,x_{C'})\geqslant 2$. We therefore have an infinite collection of points in $B(x,2)$, no subsequence of which can converge. Hence $B(x,2)$ is not compact, which implies that $T$ is not proper.
	
	Let $X$ be a proper quasi-tree, then $T_X$ is proper as the quotient map is continuous. As $X$ is $(1,C)$-quasi-isometric to $T_X$, and $T_X$ is $(1,C')$-quasi-isometric to a locally finite simplicial tree, the conclusion follows.
\end{proof}

\section{Uniform and non-uniform tree approximation}

We are now able to get a uniform version of Gromov's tree approximation in the case of quasi-trees, as an easy consequence of the results in Sections 3 and 4. In fact, we have an even stronger result, as our uniform approximation applies to any subset of our quasi-tree, as opposed to the finite subsets that were considered for general hyperbolic spaces.

\begin{prop}
	\label{Uniform}
	Let $(X,d)$ be a $\delta$-hyperbolic quasi-tree with bottleneck constant $\Delta\geqslant 0$. Let $x_0\in X$, and let $Z\subset X$. Let $Y$ be a union of geodesic segments $\bigcup_{z\in Z}[x_0,z]$. Then there exist an $\mathbb{R}$-tree $T$ and a map $f:(Y,d)\to (T,d^*)$ such that:
	\begin{enumerate}[(1)]
		\item For all $z\in Z$, the restriction of $f$ to the geodesic segment $[x_0,z]$ is an isometry.
		\item For all $x,y\in Y$, we have that $d(x,y)-2(\Delta+2\delta)\leqslant d^*(f(x),f(y))\leqslant d(x,y)$.
	\end{enumerate}
\end{prop}

\begin{proof}
	Consider the end-approximating tree $T_X$, and the end-approximating map $f:X\to T_X$. We now consider the restriction $f|_Y:Y\to T_X$. The fact that $f$ is an isometry on every geodesic segment $[x_0,z]$ was noted in the first part of the proof of \Cref{Tree}, so it is also true for $f|_Y$.
	
	For the second part, we simply apply \Cref{QuasiIsometry}. It was noted in the proof that, given $X$ is a quasi-tree, we have that $d(x,y)-2(\Delta+2\delta)\leqslant d^*([x],[y])\leqslant d(x,y)$ for every $x,y\in X$. Hence this is also true for $f|_Y$ when applied to any $x,y\in Y$.
\end{proof}

The converse to this is obviously true, as by taking our set $Z$ to be the whole of $X$ we would get a quasi-isometric embedding from $X$ to an $\mathbb{R}$-tree, giving us that $X$ is a quasi-tree by the following standard result.

\begin{defn}
	A \emph{subtree} $T$ of a metric space $X$ is a subspace of $X$ that is an $\mathbb{R}$-tree under the shortest path metric $d_T$.
\end{defn}

\begin{lem}
	\label{QIEmbedding}
	Let $X$ be a geodesic metric space. If $X$ quasi-isometrically embeds into an $\mathbb{R}$-tree $T$, then $X$ is quasi-isometric to a subtree of $T$.
\end{lem}

\begin{proof}
	Let $f:X\to T$ be an $(L,C)$-quasi-isometric embedding, where $T$ is an $\mathbb{R}$-tree. Assume that $C>0$. Let $x_0\in X$, and let $T'$ be the union of geodesic segments $\bigcup_{z\in X}[f(x_0),f(z)]$. The space $T'$ is a path-connected subset of an $\mathbb{R}$-tree, so is itself an $\mathbb{R}$-tree. We want to show that $f:X\to T'$ is a quasi-isometry.
	
	As $f:X\to T'$ is an $(L,C)$-quasi-isometric embedding, we only need to check that for every $y\in T'$, there exists $x\in X$ such that $d'(f(x),y)\leqslant C$. Suppose this is not the case, so for some $y\in T'$, we have that $B(y,C)\cap f(X)=\emptyset$. By the definition of $T'$, we can find $z\in X$ such that $y\in[f(x_0),f(z)]$.
	
	As $B(y,C)\cap f(X)=\emptyset$, we have that $f([x_0,z])$ is disconnected in $T'$, where $f(x_0)$ and $f(z)$ are in different connected components. We can group the connected components of $f([x_0,z])$ together depending on which component of $T'\backslash\{y\}$ they lie in.
	
	For every $\varepsilon>0$, we can therefore find $z_1,z_2\in[x_0,z]\subset X$ such that $d(z_1,z_2)\leqslant\varepsilon$, and $f(z_1),f(z_2)$ lie in different connected components of $f([x_0,z])$, with $y\in[f(z_1),f(z_2)]$. This tells us that $2C\leqslant d'(f(z_1),f(z_2))\leqslant Ld(z_1,z_2)+C\leqslant L\varepsilon+C$, which is a contradiction for $\varepsilon<\frac{C}{L}$. Hence $f:X\to T'$ is a quasi-isometry.
\end{proof}

Although it is sufficient by the above, the converse of \Cref{Uniform} is actually a far stronger assumption than we need to prove that $X$ is a quasi-tree. As it turns out, it is possible to show that a global quasi-isometry exists when we only assume uniform approximation for finite subsets. In particular, this means that having uniform tree approximation for finite subsets is equivalent to having uniform tree approximation for all subsets, so quasi-trees are exactly the geodesic spaces for which uniform tree approximation is possible.

In fact, to show that our space is a quasi-tree, we can even loosen our idea of uniform tree approximation for finite subsets to allow for a multiplicative error as well. This means that having this less strict version of tree approximation for all finite subsets is equivalent to having the stronger version for all finite subsets. This third equivalence was suggested by the referee of this paper.

\begin{prop}
	\label{FiniteAprrox}
	Let $(X,d)$ be a geodesic metric space. The following are equivalent:
	\begin{enumerate}[(1)]
		\item $X$ is a quasi-tree.
		\item There exists a constant $C\geqslant 0$ such that for every finite subset $Z$ of $X$, there exist an $\mathbb{R}$-tree $(T,d^*)$ and a $(1,C)$-quasi-isometric embedding $f:(Z,d)\to (T,d^*)$.
		\item There exist constants $L\geqslant 1$ and $C\geqslant 0$ such that for every finite subset $Z$ of $X$, there exist an $\mathbb{R}$-tree $(T,d^*)$ and an $(L,C)$-quasi-isometric embedding $f:(Z,d)\to (T,d^*)$. 
	\end{enumerate}
\end{prop}

\begin{proof}
	The first statement implies the second statement by \Cref{Uniform}. The second statement clearly implies the third. It remains to show that the third statement implies that $X$ is a quasi-tree.
	
	Now suppose that such constants $L\geqslant 1$ and $C\geqslant 0$ exist. Let $A=\max\{1,C\}$, so $A>0$, and note that any $(L,C)$-quasi-isometric embedding is also an $(L,A)$-quasi-isometric embedding. We want to show that $X$ satisfies the bottleneck criterion. Let $x,y\in X$, let $[x,y]$ be a geodesic between them, and let $m\in[x,y]$ be the midpoint. Let $\{a_0,\ldots,a_k\}\subset [x,y]$ be such that $a_0=x$, $a_k=y$, and $d(a_i,a_{i+1})\leqslant \frac{A}{L}$. Let $\gamma$ be a path between $x$ and $y$, and let $\{z_0,\ldots,z_n\}\subset \gamma$ be such that $z_0=x$, $z_n=y$, and $d(z_i,z_{i+1})\leqslant \frac{A}{L}$.
	
	We let $Y=\{m,a_0,\ldots,a_k,z_1,\ldots,z_{n-1}\}$, and so, by our assumption, there exist an $\mathbb{R}$-tree $T$ and a map $f:(Y,d)\to (T,d^*)$ such that $\frac{1}{L}d(a,b)-A\leqslant d^*(f(a),f(b))\leqslant Ld(a,b)+A$ for every $a,b\in Y$.
	
	Consider $[f(x),f(y)]$. As $T$ is an $\mathbb{R}$-tree, there exists $m'\in[f(x),f(y)]$ such that $d^*(f(x),f(m))$ $=d^*(f(x),m')+d^*(m',f(m))$ and $d^*(f(m),f(y))=d^*(f(m),m')+d^*(m',f(y))$. In other words, $m'$ is the closest point in $[f(x),f(y)]$ to $f(m)$.
	
	Consider the path formed by the concatenation $[f(a_0),f(a_1)]*\cdots *[f(a_{k-1}),f(a_k)]$. This is a path from $f(x)$ to $f(y)$ in $T$, so $m'\in [f(a_i),f(a_{i+1})]$ for some $i$. Recall that $d(a_i,a_{i+1})\leqslant \frac{A}{L}$, so $d^*(f(a_i),f(a_{i+1}))\leqslant 2A$. Let $j\in\{1,\ldots,n\}$ be the least element such that $f(a_j)\in B(m',2A)$, and let $l\in\{1,\ldots,n\}$ be the largest element such that $f(a_l)\in B(m',2A)$. The preceding statements imply that these exist, and that $j\neq l$.
	
	We note that $f(a_j)$ lies in the same connected component of $T\backslash \{m'\}$ as $f(x)$, as either $a_j=x$ or $f(a_{j-1})$ lies in the same connected component of $T\backslash B(m',2A)$ as $f(x)$. By the same reasoning, $f(a_l)$ lies in the same connected component of $T\backslash \{m'\}$ as $f(y)$, so $m'$ separates $f(a_j)$ and $f(a_l)$ in $T$, and therefore $m'\in [f(a_j),f(a_l)]$.
	
	We also want to show that $m\in [a_j,a_l]$. Note that $m\in[a_i,a_{i+1}]$ for some $i$, and both $d^*(f(a_i),f(m))\leqslant 2A$ and $d^*(f(a_{i+1}),f(m))\leqslant 2A$. Either $f(m)=m'$, or $f(m)$ lies in a different connected component of $T\backslash \{m'\}$ to $f(x)$ and to $f(y)$. Either way, this implies that $i\geqslant j$ and $i+1\leqslant l$, so $m\in [a_j,a_l]$.
	
	We therefore have that
	\begin{align*}
		d^*(f(a_j),f(a_l))+2d^*(m',f(m)) & =d^*(f(a_j),f(m))+d^*(f(m),f(a_l))
		\\ & \leqslant Ld(a_j,m)+Ld(m,a_l)+2A
		\\ & =Ld(a_j,a_l)+2A
		\\ & \leqslant L^2d^*(f(a_j),f(a_l))+L^2A+2A.
	\end{align*}
	This means that
	\begin{align*}
		d^*(m',f(m)) & \leqslant \frac{1}{2}((L^2-1)d^*(f(a_j),f(a_l))+L^2A+2A)
		\\ & \leqslant \frac{1}{2}(4(L^2-1)A+L^2A+2A)
		\\ & =3L^2A-A.
	\end{align*}
	
	Consider the path formed by the concatenation $[f(z_0),f(z_1)]*\cdots *[f(z_{n-1}),f(z_n)]$. This is a path from $f(x)$ to $f(y)$ in $T$, so as before $m'\in [f(z_i),f(z_{i+1})]$ for some $i$. Note that $d^*(f(z_i),f(z_{i+1}))\leqslant Ld(z_i,z_{i+1})+A\leqslant 2A$, so for some $i$ we have that $d^*(f(z_i),m')\leqslant A$.
	
	We conclude that
	\begin{align*}
		d(m,z_i) & \leqslant Ld^*(f(m),f(z_i))+A
		\\ & \leqslant L(d^*(f(m),m')+d^*(m',f(z_i)))+A
		\\ & \leqslant 3L^3A+A,
	\end{align*}
	so $X$ satisfies the bottleneck criterion, and therefore is a quasi-tree.
\end{proof}

Gromov also proved an alternative version of tree approximation in hyperbolic spaces, in which the approximating tree is a subspace of the hyperbolic space.

\begin{prop}
	\label{SubtreeApproximation}
	\emph{\cite[p.~157--158]{Gromov1987}\cite{Bowditch2006}}
	Suppose $X$ is a $\delta$-hyperbolic geodesic metric space. There exists a function $h:\mathbb{N}\to[0,\infty)$ such that if $Y\subset X$ is a finite set of points with $|Y|=n$, then there exists a subtree $T\subset X$ such that $Y\subset T$ and for every $x,y\in T$, we have that $d_T(x,y)\leqslant d(x,y)+\delta h(n)$.
\end{prop}

It is clear from the definition of $d_T$ that $d(x,y)\leqslant d_T(x,y)$. In this version of tree approximation, the $\mathbb{R}$-tree $T$ is constructed in the obvious way: let $Y=\{y_1,\ldots,y_n\}$, set $T_1=[y_1,y_2]$, then set $T_2=[y_3,t]\cup T_1$, where $t$ is a closest point in $T_1$ to $y_3$, and so on. This process will terminate at $T=T_{n-1}$.

It is natural to ask if there is a uniform version of this type of tree approximation in the case of quasi-trees.

\begin{ques}
	\label{BowditchQuestion}
	Let $X$ be a quasi-tree. Does there exist a constant $C\geqslant 0$, depending only on $X$, such that for any set of finite points $Y\subset X$, we can construct an $\mathbb{R}$-tree $T\subset X$ such that $Y\subset T$, and for every $x,y\in T$, we have that $d_T(x,y)\leqslant d(x,y)+C$?
\end{ques}

It is immediately obvious that taking the same approach as in the proof of \Cref{SubtreeApproximation} will not work, as the following example shows.

\begin{exm}
	Consider the space $X=[0,1]\times\mathbb{R}$ with the Euclidean metric. It is clear that $X$ is a quasi-tree. Let $n\in\mathbb{N}$ be odd, and let $Y=\{(0,0),(1,1),(0,2),\ldots,$ $(0,n-1),(1,n)\}$. If a tree $T$ is constructed using this ordering of the points, we get that $d((0,0),(1,n))=\sqrt{1+n^2}$ and $d_T((0,0),(1,n))=n\sqrt{2}$, so
	\begin{equation*}
	\lim_{n\to\infty}\frac{d((0,0),(1,n))}{d_T((0,0),(1,n))}=\frac{1}{\sqrt{2}}.
	\end{equation*}
	In particular, we cannot have that $d_T((0,0),(1,n))\leqslant d((0,0),(1,n))+C$ for all odd $n\in\mathbb{N}$, whatever the choice of the constant $C$.
\end{exm}

Although the naive approach does not work, in the above example the answer to this question is actually yes, which we can easily see by letting $T=(\{\frac{1}{2}\}\times \mathbb{R})\cup\bigcup_{i=1}^n([0,1]\times\{y_i\})$, and letting $C=1$. We can however show that there are quasi-trees in which the answer to the above question will always be no, whichever way the subtree is constructed. The idea in the following example is that, over a large scale, it is not possible to keep the paths in a subtree tight to geodesics in all directions.

\begin{exm}
	\label{NonUniformExample}
	Take the infinite 4-regular tree with a base vertex $x_0$ and edges of length 1. For all $n\in\mathbb{Z}_{\geqslant 0}$, replace the vertices in $S(x_0,n)$ with copies of the rectangle $[0,2]\times[0,2^{n+1}]$. Attach the edges to these vertex rectangles at the midpoint of the sides, with one of the short sides oriented towards $[0,2]^2$ (see \Cref{FigNonUniformApproxBW2} for an illustration). We will call this space $X$, and note that it is a quasi-tree under the piecewise Euclidean metric by Manning's bottleneck criterion.
	
	\begin{figure}[h]
		\centering
		\includegraphics[width=\textwidth]{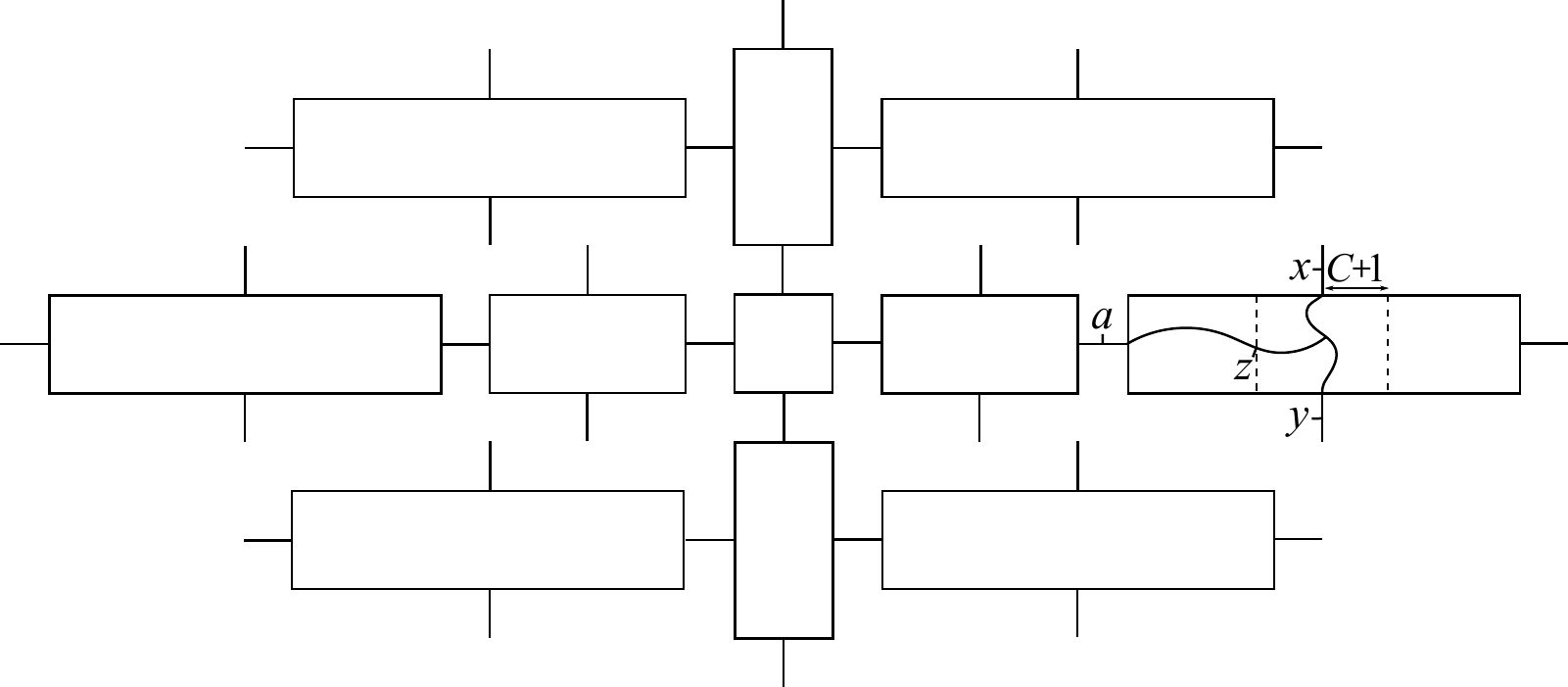}
		\caption{A counterexample to uniform subtree approximation}
		\label{FigNonUniformApproxBW2}
	\end{figure}
	
	Let $C>0$ be arbitrary. Let $n\in\mathbb{N}$ be such that $2^{n+1}>2C+3$, and let $Y_n$ be the set of midpoints of the edges in $X$ that correspond to the edges in $B(x_0,n)$ in the 4-regular tree. Suppose that $T$ is a subspace of $X$ such that $Y_n\subset T$, $T$ is a tree under the shortest path metric $d_T$, and $d_T(x,y)\leqslant d(x,y)+C$ for every $x,y\in T$. Consider a rectangle in $X$ of size $[0,2]\times[0,2^{m+1}]$, where $2C+2<2^{m+1}<2^{n+1}$.
	
	Let $x,y\in Y_n$ be the midpoints of the edges attached to the long sides of this rectangle, and let $a\in Y_n$ be the midpoint of the edge attached to the short side of this rectangle nearest to the central square. We will show that at least one of $d_T(a,x)$ or $d_T(a,y)$ must have large error in comparison to $d(a,x)$ or $d(a,y)$, respectively.
	
	We first note that there must be a unique shortest path in $T$ that connects $x$ to $y$, and it must have length at most $d(x,y)+C=3+C$ by our assumption on $T$. Therefore this path must stay in the subrectangle $[0,2]\times[2^m-(C+1),2^m+C+1]$, as otherwise $d_T(x,y)\geqslant 1+2(C+1)= 3+2C>d(x,y)+C$.
	
	As $2^m>C+1$, this subrectangle is a proper subset of the original rectangle, and, in particular, does not contain the attaching point of the edge containing $a$. As $T$ is a tree, this means that outside this subrectangle the shortest path in $T$ between $a$ and $x$ must be identical to the shortest path in $T$ between $a$ and $y$. Let $z$ be the point at which this path intersects $[0,2]\times\{2^m-(C+1)\}$, and suppose without loss of generality that $d(x,z)\geqslant d(y,z)$.
	
	We now want to use this to compare $d_T(a,x)$ to $d(a,x)$. We first note that
	\begin{equation*}
	d(a,z)+d(z,x)\leqslant d_T(a,z)+d_T(z,x)=d_T(a,x).
	\end{equation*}
	For $z'\in[0,2]\times\{2^m-(C+1)\}$ satisfying $d(x,z')\geqslant d(y,z')$, the sum $d(a,z')+d(z',x)$ will be minimal when $z'=(1,2^m-(C+1))$, which is the central point of that side of the subrectangle. We therefore get that
	\begin{equation*}
	d_T(a,x)\geqslant d(a,z)+d(z,x)\geqslant 2^m-(C+1)+\frac{1}{2}+\sqrt{1+(C+1)^2}+\frac{1}{2},
	\end{equation*}
	while $d(a,x)=1+\sqrt{1+(2^m)^2}$, so
	\begin{equation*}
	d_T(a,x)-d(a,x)\geqslant \sqrt{1+(C+1)^2}-(C+1)+2^m-\sqrt{1+(2^m)^2}.
	\end{equation*}
	Note that
	\begin{equation*}
	2^m-\sqrt{1+(2^m)^2}\to 0 \text{ as } m\to\infty,
	\end{equation*} so taking $\varepsilon=\frac{\sqrt{1+(C+1)^2}-(C+1)}{2}$, we know that there exists some $M\in\mathbb{N}$ such that $2^M>C+1$, and such that for all choices of $m$, where $n>m\geqslant M$, we have that $d_T(a,x)\geqslant d(a,x)+\varepsilon$. This is true for any approximating tree $T\subset X$ of $Y_n$, and in fact $M$ can be chosen uniformly, as it is not dependent on $n$.
	
	We now want to use this to show that, if we choose large enough $n$, no subtree of $X$ can approximate $Y_n$ with error no greater than $C$. Suppose otherwise, so suppose there exists a constant $C>0$ such that for any set of finite points $Y\subset X$, there exists an $\mathbb{R}$-tree $T$ that is embedded in $X$ such that $Y\subset T$, and such that for every $x,y\in T$, we have that $d_T(x,y)\leqslant d(x,y)+C$. Let $n\in\mathbb{N}$ be such that $\varepsilon(n-M-1)>C$, and consider $Y_n$, with $T$ the associated subtree of $X$.
	
	Let $x_1$ be the midpoint of an edge adjacent to the central square $[0,2]^2$. For $x_i$, where $1\leqslant i\leqslant n-2$, we choose $x_{i+1}\in Y_n$ such that $x_i$ and $x_{i+1}$ are midpoints of edges attached to the same rectangle, $x_{i+1}$ is strictly further from the central square than $x_i$, and $x_{i+1}$ is the midpoint of an edge attached to one of the long sides of the shared rectangle. We can think of this as moving outwards in the quasi-tree $X$, turning ``left'' or ``right'' each time we enter a rectangle. This gives us two possible options for $x_{i+1}$, we take the option that maximises $d_T(x_i,x_{i+1})$.
	
	This means that for $i\geqslant M$, we have that $d_T(x_i,x_{i+1})\geqslant d(x_i,x_{i+1})+\varepsilon$, and we note that for all other $i$, we still have that $d_T(x_i,x_{i+1})\geqslant d(x_i,x_{i+1})$. By our construction of $X$, any path between $x_1$ and $x_{n-1}$ must pass through every $x_i$, so we can see that $d_T(x_1,x_{n-1})=d_T(x_1,x_2)+\cdots+d_T(x_{n-2},x_{n-1})$. It is also clear from our construction of $X$ that $d(x_1,x_{n-1})=d(x_1,x_2)+\cdots+d(x_{n-2},x_{n-1})$, so
	\begin{equation*}
	d_T(x_1,x_{n-1})\geqslant d(x_1,x_{n-1}) + \varepsilon(n-M-1)>d(x_1,x_{n-1})+C.
	\end{equation*}
	This is a contradiction, so no such constant $C>0$ exists for the quasi-tree $X$.
\end{exm}

Although the answer to \Cref{BowditchQuestion} is no in general, we can still ask the following question.

\begin{ques}
	\label{BowditchCharacterisation}
	Does there exist a characterisation of those quasi-trees for which the answer to \Cref{BowditchQuestion} is yes?
\end{ques}

\section{Geometry of the end-approximating tree}

We will now return to the end-approximating tree $T_X$. In Section 3, we showed that if $X$ is a geodesic metric space, then $T_X$ is an $\mathbb{R}$-tree, and in Section 4, we showed that if $X$ is a quasi-tree, then $T_X$ is $(1,C)$-quasi-isometric to $X$. In this section, we will give more detail on how the geometry of $X$ relates to the geometry of $T_X$. In particular, we will show that when $X$ is proper their space of ends are homeomorphic, and that when $X$ is a quasi-tree, their boundaries can be thought of as being isometric.

\subsection{Path components}

Here we will formalise the idea that $T_X$ is constructed by collapsing $X$ along the spheres $S(x_0,r)$, which was mentioned at the end of Section 3. In doing so we will look at the correspondence between the path components of $X\backslash B(x_0,r)$ and $T_X\backslash B([x_0],r)$, which will be relevant in Sections 6.2 and 6.3.

For this section, we will assume a fixed basepoint $x_0\in X$. We begin with the following easy lemma.

\begin{lem}
	\label{Paths}
	Let $(X,d)$ be a metric space, and let $x,y\in X$. Suppose that for every $\varepsilon>0$, there exists a path $\gamma:[0,1]\to X$ from $x$ to $y$ such that $d(x_0,\gamma(t))\geqslant\min\{d(x_0,x),d(x_0,y)\}-\varepsilon$ for all $t\in[0,1]$. Then $d^*([x],[y])=|d(x_0,x)-d(x_0,y)|$.
\end{lem}

\begin{proof}
	Recall that $d^*([x],[y])=d'(x,y)=d(x_0,x)+d(x_0,y)-2(x,y)'_{x_0}$, so if we can show that $(x,y)'_{x_0}=\min\{d(x_0,x),d(x_0,y)\}$, then we will be done. We already know that $(x,y)'_{x_0}\leqslant\min\{d(x_0,x),d(x_0,y)\}$, so we just need the opposite inequality.
	
	Let $\varepsilon>0$, and let $\gamma:[0,1]\to X$ be a path from $x$ to $y$ such that $d(x_0,\gamma(t))\geqslant\min\{d(x_0,x),$ $d(x_0,y)\}-\varepsilon$ for all $t\in[0,1]$. Let $\nu>0$, and choose a sequence $0=t_1<t_2<\cdots<t_n=1$ such that $d(\gamma(t_i),\gamma(t_{i+1}))$ $\leqslant \nu$ for every $1\leqslant i\leqslant n-1$. Then
	\begin{align*}
	(\gamma(t_i),\gamma(t_{i+1}))_{x_0} & =\frac{1}{2}(d(x_0,\gamma(t_i))+d(x_0,\gamma(t_{i+1}))-d(\gamma(t_i),\gamma(t_{i+1})))
	\\ & \geqslant \frac{1}{2}(2\min\{d(x_0,x),d(x_0,y)\}-2\varepsilon-\nu)
	\\ & =\min\{d(x_0,x),d(x_0,y)\}-\varepsilon-\frac{\nu}{2}.
	\end{align*}
	It follows that $(x,y)'_{x_0}\geqslant\min\{d(x_0,x),d(x_0,y)\}-\varepsilon-\frac{\nu}{2}$, so as $\varepsilon>0$ and $\nu>0$ were arbitrary, we can conclude that $(x,y)'_{x_0}\geqslant\min\{d(x_0,x),d(x_0,y)\}$, and we therefore have equality.
\end{proof}

\begin{cor}
	\label{OneEndedExample}
	If $X$ is a geodesic metric space such that for every $r\in[0,\infty)$ and every $x,y\in S(x_0,r)$ there exists a path $\gamma$ between $x$ and $y$ such that $d(x_0,\gamma)\geqslant r$, then $T_X$ is isometric to $[0,\infty)$.
\end{cor}

We would like to show that the reverse implication to \Cref{Paths} also holds. Note that the $\varepsilon>0$ is necessary for this to be the case. For example, if we consider $X$ to be an open ball around some $x_0$ in $\mathbb{R}^2$ with two points $x$ and $y$ added from its boundary, then it can be seen that those boundary points would still collapse together in $T_X$ despite the lack of any path such that $d(x_0,\gamma(t))\geqslant\min\{d(x_0,x),d(x_0,y)\}$.

To show the reverse implication to \Cref{Paths}, we will need to look at the relationship between the path components of $X\backslash B(x_0,r)$ and $T_X\backslash B([x_0],r)$.

\begin{lem}
	\label{PathComponents}
	Let $(X,d)$ be a geodesic metric space. Let $x,y\in X$, and let $r\in[0,\infty)$. Then $x$ and $y$ lie in the same path component of $X\backslash B(x_0,r)$ if and only if $([x],[y])^*_{[x_0]}> r$.
\end{lem}

\begin{proof}
	Suppose $x$ and $y$ lie in the same path component of $X\backslash B(x_0,r)$. Then there exist $\delta>0$ and a path $\gamma:[0,1]\to X$ between them such that $d(x_0,\gamma(t))\geqslant r+\delta$ for all $t\in[0,1]$. The existence of such a $\delta$ follows from the fact that $\gamma$ is continuous, as this implies that $\gamma([0,1])$ is a compact subset of $X\backslash B(x_0,r)$. If there existed a sequence of points $(x_n)$ in $\gamma([0,1])$ such that $d(x_0,x_n)\to r$, then there would exist $x\in\gamma([0,1])$ such that $d(x_0,x)= r$, a contradiction. Therefore we can apply the same argument as in the proof of \Cref{Paths} to get that $([x],[y])^*_{[x_0]}=(x,y)'_{x_0}\geqslant r+\delta>r$.
	
	Now suppose $x$ and $y$ do not lie in the same path component of $X\backslash B(x_0,r)$. If $x$ or $y$ lies in $B(x_0,r)$, then it is immediate that
	\begin{equation*}
	([x],[y])^*_{[x_0]}=(x,y)'_{x_0}\leqslant \min\{d(x_0,x),d(x_0,y)\}\leqslant r.
	\end{equation*}
	If $x$ and $y$ lie in different path components of $X\backslash B(x_0,r)$, then any geodesic $[x,y]$ must intersect $B(x_0,r)$. Let $z\in[x,y]\cap B(x_0,r)$. As $d(x_0,z)\leqslant r$, by \Cref{Aux2} we have that $(x,y)_{x_0} \leqslant r$.
	
	Let $(x_1,\ldots,x_n)$ be any sequence in $S_{x,y}$. As $x_1=x$ and $x_n=y$, and $x$ and $y$ lie in different path components of $X\backslash B(x_0,r)$, there must be some pair $x_i,x_{i+1}$ such that $x_i\in B(x_0,r)$, or $x_{i+1}\in B(x_0,r)$, or $x_i$ and $x_{i+1}$ lie in different path components of $X\backslash B(x_0,r)$. Therefore $(x_i,x_{i+1})_{x_0}\leqslant r$, so as this sequence in $S_{x,y}$ was arbitrary we must have that $([x],[y])^*_{[x_0]}=(x,y)'_{x_0}\leqslant r$.
\end{proof}

\begin{rem}
	This would not hold in general if we instead removed an open ball and considered when $([x],[y])^*_{[x_0]}\geqslant r$. To see this, we can again look at the example of $X$ being an open ball in $\mathbb{R}^2$ with two points added from its boundary.
\end{rem}

\begin{cor}
	\label{PathCompCor}
	Let $(X,d)$ be a geodesic metric space. Let $x,y\in X$, and let $r\in[0,\infty)$. We then have that $x$ and $y$ lie in the same path component of $X\backslash B(x_0,r)$ if and only if $[x]$ and $[y]$ lie in the same path component of $T_X\backslash B([x_0],r)$.
\end{cor}

\begin{proof}
	As $T_X$ is a geodesic tree, $[x]$ and $[y]$ lie in the same path component of $T_X\backslash B([x_0],r)$ if and only if $([x],[y])^*_{[x_0]}> r$, which by \Cref{PathComponents} is true if and only if $x$ and $y$ lie in the same path component of $X\backslash B(x_0,r)$.
\end{proof}

We can now give an alternative definition of the metric $d^*$ from the geometry of $X$.

\begin{cor}
	\label{AltGromovProduct}
	Let $X$ be a geodesic metric space, and let $x,y\in X$. Then $([x],[y])^*_{[x_0]}=\sup\{r\in[0,\infty):x\text{ and }y\text{ lie in the same path component of }$ $X\backslash B(x_0,r)\}$.
\end{cor}

\begin{proof}
	This follows from the fact that $T_X$ is a tree, so $([x],[y])^*_{[x_0]}=d^*([x_0],[z])$, where $[z]\in T_X$ is the unique point such that $[[x_0],[z]]=[[x_0],[x]]\cap[[x_0],[y]]$. We then simply have to notice that $d^*([x_0],[z])=\sup\{r\in[0,\infty):[x]\text{ and }[y]\text{ lie in the same path component of }$ $T_X\backslash B([x_0],r)\}$, and apply \Cref{PathCompCor}.
\end{proof}

\begin{rem}
	\label{XRTree}
	If $X$ is an $\mathbb{R}$-tree, then this confirms that $T_X=X$, as \Cref{AltGromovProduct} tells us that $(x,y)_{x_0}=([x],[y])^*_{[x_0]}$, so $d(x,y)=d^*([x],[y])$.
\end{rem}

\begin{rem}
	Recall that $([x],[y])^*_{[x_0]}=(x,y)'_{x_0}$. We can use this along with \Cref{Aux1} and \Cref{Aux2} to give an alternative proof of the fact that when $X$ is a quasi-tree, the inequality $(x,y)_{x_0}\leqslant (x,y)'_{x_0}\leqslant (x,y)_{x_0}+\Delta+2\delta$ holds, where $\Delta\geqslant 0$ is the bottleneck constant and $\delta\geqslant 0$ is the hyperbolicity constant.
\end{rem}

We can now obtain our improvement of \Cref{Paths}.

\begin{cor}
	\label{PathsImproved}
	Let $(X,d)$ be a geodesic metric space, and let $x,y\in X$. We have that $d^*([x],[y])=|d(x_0,x)-d(x_0,y)|$ if and only if for every $\varepsilon>0$ there exists a path $\gamma:[0,1]\to X$ from $x$ to $y$ such that $d(x_0,\gamma(t))\geqslant\min\{d(x_0,x),d(x_0,y)\}-\varepsilon$ for all $t\in[0,1]$.
\end{cor}

\begin{proof}
	If $d^*([x],[y])=|d(x_0,x)-d(x_0,y)|$, then $([x],[y])^*_{[x_0]}=\min\{d(x_0,x),d(x_0,y)\}$, so the forward implication follows from \Cref{AltGromovProduct}. The reverse implication is \Cref{Paths}.
\end{proof}

As a particular case of this, we can see that two points in $X$ will collapse together in $T_X$ if and only if they lie on the same sphere and there exist appropriate paths between them, as mentioned at the end of Section 3.

\begin{cor}
	\label{Collapse}
	Let $(X,d)$ be a geodesic metric space, and let $x,y\in X$. We have that $[x]=[y]$ if and only if there exists $r\in [0,\infty)$ such that $x,y\in S(x_0,r)$, and for every $\varepsilon>0$, there exists a path $\gamma$ between them in $X$ such that $d(x_0,\gamma)\geqslant r-\varepsilon$.
\end{cor}

\begin{proof}
	Suppose $[x]=[y]$. Then $d^*([x_0],[x])=d^*([x_0],[y])$, so $d(x_0,x)=d(x_0,y)$, and so $d^*([x],[y])=|d(x_0,x)-d(x_0,y)|$. Both directions now follow from \Cref{PathsImproved}.
\end{proof}

\begin{rem}
	Suppose that two geodesic rays $\gamma_1:[0,\infty)\to X$ and $\gamma_2:[0,\infty)\to X$ based at $x_0$ have the property that for every $t\in [0,\infty)$ and $\varepsilon>0$, the points $\gamma_1(t)$ and $\gamma_2(t)$ are in the same path component of $X\backslash B(x_0,t-\varepsilon)$. Then \Cref{Collapse} means that $[\gamma_1(t)]=[\gamma_2(t)]$ for every $t\in[0,\infty)$, and, in particular, $\gamma_1$ and $\gamma_2$ are sent to the same geodesic ray in $T_X$.
\end{rem}

\subsection{Ends of $T_X$ for proper metric spaces}

In this subsection we will show that, when $X$ is a proper metric space, we can think of $T_X$ as $X$ with each of its ends collapsed into one geodesic ray. This justifies our description of $T_X$ as the end-approximating tree of a metric space. In particular, we will show that the space of ends of $X$ is homeomorphic to the space of ends of $T_X$, and as a result is homeomorphic to the space of ends of a locally finite simplicial tree. These results are somewhat obvious, but we include them for completeness.

We will define the space of ends for a proper geodesic metric space, for the more general definition in topological spaces see \cite[p.~144]{Bridson1999}. As before, we will assume a fixed basepoint $x_0\in X$.

\begin{defn}
	Let $X$ be a proper metric space. A \emph{proper ray} in $X$ is a continuous map $\gamma:[0,\infty)\to X$ such that for every $r\in[0,\infty)$, the set $\gamma^{-1}(B(x_0,r))$ is bounded.
\end{defn}

\begin{defn}
	Let $X$ be a proper metric space, and let $\gamma_1$ and $\gamma_2$ be proper rays in $X$. We say that $\gamma_1$ and $\gamma_2$ \emph{converge to the same end} if for every $r\in[0,\infty)$, there exists $t\in[0,\infty)$ such that $\gamma_1([t,\infty))$ and $\gamma_2([t,\infty))$ are contained in the same path component of $X\backslash B(x_0,r)$. This defines an equivalence relation on the proper rays in $X$. We denote the equivalence class of a proper ray $\gamma$ by $\text{end}(\gamma)$.
\end{defn}

The following lemma can be found in \cite{Bridson1999}. We will mimic its proof for \Cref{Surjective}.

\begin{lem}
	\emph{\cite[Lemma~I.8.28]{Bridson1999}}
	\label{GeodesicRep}
	Let $X$ be a proper geodesic metric space. For every proper ray $\gamma$ in $X$, there exists a geodesic ray $\gamma'$ based at $x_0\in X$ such that $\gamma'\in\emph{end}(\gamma)$.
\end{lem}

The proof of the above lemma relies on the Arzel\`{a}-Ascoli theorem, we state the necessary version of it here. For a stronger version, see \cite[p.~290]{Munkres2000}.

\begin{thm}[Arzel\`{a}-Ascoli theorem]
	\label{AA}
	Let $X$ and $Y$ be metric spaces, and let $\mathcal{F}$ be a subset of $\{f:X\to Y:f\text{ continuous}\}$. If $\mathcal{F}$ is equicontinuous, and the set $\mathcal{F}_x=\{f(x):f\in\mathcal{F}\}$ has compact closure in $Y$ for every $x\in X$, then a subsequence of $\mathcal{F}$ converges pointwise to a continuous function $f:X\to Y$.
\end{thm}

The space of ends, and the topology on it, are defined as follows.

\begin{defn}
	Let $X$ be a proper metric space. The \emph{space of ends} of $X$ is the set of equivalence classes of proper rays in $X$, denoted $\text{Ends}(X)$. A sequence of ends converges, $\text{end}(\gamma_n)\to\text{end}(\gamma)$, if for every $r\in[0,\infty)$ there exists a sequence of real numbers $(N_n)_{n\in\mathbb{N}}$ such that $\gamma_n([N_n,\infty))$ and $\gamma([N_n,\infty))$ lie in the same path component of $X\backslash B(x_0,r)$ for all $n\in\mathbb{N}$ sufficiently large. The topology on $\text{Ends}(X)$ is induced by this convergence.
\end{defn}

If two proper metric spaces $X$ and $Y$ are quasi-isometric, then it is well known that this quasi-isometry induces a homeomorphism between $\text{Ends}(X)$ and $\text{Ends}(Y)$ \cite[Proposition~I.8.29]{Bridson1999}. This means that when $X$ is a quasi-tree, we automatically get that $\text{Ends}(X)$ and $\text{Ends}(T_X)$ will be homeomorphic. We will show that this is in fact true for every proper metric space $X$.

\begin{lem}
	Let $(X,d)$ be a proper geodesic metric space. Let $f:X\to T_X$ be the end-approximating map. For every proper ray $\gamma:[0,\infty)\to X$, the map $f\circ\gamma:[0,\infty)\to T_X$ is a proper ray.
\end{lem}

\begin{proof}
	We first recall that, as $d^*([x],[y])\leqslant d(x,y)$, the function $f$ is continuous. Therefore $f\circ\gamma$ is continuous. Now let $r\in [0,\infty)$, and note that $f^{-1}(B([x_0],r))=B(x_0,r)$. As $\gamma$ is a proper ray, $\gamma^{-1}(f^{-1}(B([x_0],r)))=\gamma^{-1}(B(x_0,r))$ is bounded, so we can conclude that $f\circ\gamma$ is a proper ray.
\end{proof}

\begin{cor}
	\label{Injective}
	Let $(X,d)$ be a proper geodesic metric space. Let $f:X\to T_X$ be the end-approximating map. Let $\gamma_1:[0,\infty)\to X$ and $\gamma_2:[0,\infty)\to X$ be proper rays in $X$. Then $\gamma_2\in\emph{end}(\gamma_1)$ if and only if $f\circ\gamma_2\in\emph{end}(f\circ\gamma_1)$.
\end{cor}

\begin{proof}
	We have that $\gamma_2\in\text{end}(\gamma_1)$ if and only if for every $r\in[0,\infty)$, there exists $t\in[0,\infty)$ such that $\gamma_1([t,\infty))$ and $\gamma_2([t,\infty))$ lie in the same path component of $X\backslash B(x_0,r)$. By \Cref{PathCompCor}, this is true if and only if for every $r\in[0,\infty)$, there exists $t\in[0,\infty)$ such that $f(\gamma_1([t,\infty)))$ and $f(\gamma_2([t,\infty)))$ lie in the same path component of $T_X\backslash B([x_0],r)$, and this is equivalent to saying that $f\circ\gamma_2\in\text{end}(f\circ\gamma_1)$.
\end{proof}

The proof of the following lemma uses the same idea as the proof of \Cref{GeodesicRep}, as found in \cite{Bridson1999}.

\begin{lem}
	\label{Surjective}
	Let $(X,d)$ be a proper geodesic metric space. Let $f:X\to T_X$ be the end-approximating map. Let $\gamma:[0,\infty)\to T_X$ be a proper ray in $T_X$. Then there exists a geodesic ray $\xi:[0,\infty)\to X$ such that $f\circ\xi\in\emph{end}(\gamma)$.
\end{lem}

\begin{proof}
	We know by \Cref{GeodesicRep} that there exists a geodesic representative $\gamma'\in \text{end}(\gamma)$. For every $n\in\mathbb{N}$, we choose $x_n\in f^{-1}(\gamma'(n))$, then define $\xi_n:[0,\infty)\to X$ by extending the geodesic $[x_0,x_n]$ such that $\xi_n$ is constant on $[d(x_0,x_n),\infty)$. Note that this collection of functions is equicontinuous.
	
	We can also note that for any $t\in[0,\infty)$, we have that $\{\xi_n(t):n\in\mathbb{N}\}\subset B(x_0,t)$, so as $X$ is proper the closure is compact. By the Arzel\`{a}-Ascoli theorem, there exists a subsequence $(\xi_{n_k})_{k\in\mathbb{N}}$ of $(\xi_n)_{n\in\mathbb{N}}$ that converges pointwise to a continuous function $\xi:[0,\infty)\to X$. We want to show that $\xi$ is geodesic and that $f\circ\xi=\gamma'$. The fact that $\xi$ is geodesic follows immediately from the fact that for every $t\in[0,\infty)$, we have that $d(x_0,\xi(t))=\lim_{k\to\infty}d(x_0,\xi_{n_k}(t))$, and that $d(x_0,\xi_{n_k}(t))=t$ for large enough $k$.
	
	As $f$ is continuous, we also get that $f\circ\xi(t)=\lim_{k\to\infty}f\circ\xi_{n_k}(t)$. For $k\in\mathbb{N}$ such that $n_k\geqslant t$, we have that $\xi_{n_k}(t)$ lies on the geodesic $[x_0,x_{n_k}]$, where $x_{n_k}\in f^{-1}(\gamma'(n_k))$. Hence $f\circ\xi_{n_k}(t)$ lies on the geodesic $[[x_0],\gamma'(n)]$ in $T_X$. As $d^*([x_0],f\circ\xi_{n_k}(t))=t$, and $T_X$ is a tree, we must have that $f\circ\xi_{n_k}(t)=\gamma'(t)$. Therefore $f\circ\xi(t)=\gamma'(t)$ for all $t\in[0,\infty)$, so $f\circ\xi=\gamma'$. We conclude that $\xi$ is a geodesic ray such that $f\circ\xi\in\text{end}(\gamma)$.
\end{proof}

\begin{prop}
	\label{EndsHomeomorphic}
	Let $(X,d)$ be a proper geodesic metric space. Let $f:X\to T_X$ be the end-approximating map, and let $F:\emph{Ends}(X)\to\emph{Ends}(T_X)$ be defined by $F(\emph{end}(\gamma))=\emph{end}(f\circ\gamma)$ for every $\emph{end}(\gamma)\in\emph{Ends}(X)$. The map $F$ is a homeomorphism.
\end{prop}

\begin{proof}
	The fact that $F$ is well-defined and injective comes from \Cref{Injective}, and the fact that it is surjective comes from \Cref{Surjective}. It remains to show that $F$ and $F^{-1}$ are continuous.
	
	Let $V\subset\text{Ends}(T_X)$ be closed. We want to show that $F^{-1}(V)$ is closed. Let $(\text{end}(\xi_n))_{n\in\mathbb{N}}$ be a sequence in $F^{-1}(V)$ such that $\text{end}(\xi_n)\to\text{end}(\xi)$, and choose these representatives to be geodesic rays. We want to show that $\text{end}(\xi)\in F^{-1}(V)$.
	
	We first note that $(\text{end}(f\circ\xi_n))_{n\in\mathbb{N}}$ gives a sequence in $V$. We want to show that $\text{end}(f\circ\xi_n)\to\text{end}(f\circ \xi)$. Let $r\in[0,\infty)$. As $\text{end}(\xi_n)\to\text{end}(\xi)$, there exists a sequence of natural numbers $(N_n)_{n\in\mathbb{N}}$ such that $\xi_n([N_n,\infty))$ and $\xi([N_n,\infty))$ lie in the same path component of $X\backslash B(x_0,r)$ for large enough $n$. Clearly $f\circ\xi_n([N_n,\infty))$ and $f\circ\xi([N_n,\infty))$ lie in  $T_X\backslash B([x_0],r)$ for large enough $n$, and by \Cref{PathCompCor}, they must lie in the same path component. Hence $\text{end}(f\circ\xi_n)\to\text{end}(f\circ \xi)$, so as $V$ is closed, we must have that $\text{end}(f\circ \xi)\in V$, so $\text{end}(\xi)\in F^{-1}(V)$, so $F^{-1}(V)$ is closed. Therefore $F$ is continuous.
	
	The argument for $F^{-1}$ is almost identical. Let $U\subset\text{Ends}(X)$ be closed. We want to show that $F(U)$ is closed. Let $(\text{end}(\gamma_n))_{n\in\mathbb{N}}$ be a sequence in $F(U)$ such that $\text{end}(\gamma_n)\to\text{end}(\gamma)$, and choose these representatives to be geodesic rays. We want to show that $\text{end}(\gamma)\in F(U)$.
	
	We pick geodesic rays $\xi_n,\xi:[0,\infty)\to X$ such that $\text{end}(f\circ\xi_n)=\text{end}(\gamma_n)$ and $\text{end}(f\circ\xi)=\text{end}(\gamma)$, and note that $(\text{end}(\xi_n))_{n\in\mathbb{N}}$ gives a sequence in $U$, as $F$ is bijective. We want to show that $\text{end}(\xi_n)\to\text{end}(\xi)$. Let $r\in[0,\infty)$. As $\text{end}(f\circ\xi_n)\to\text{end}(f\circ\xi)$, there exists a sequence of natural numbers $(N_n)_{n\in\mathbb{N}}$ such that $f\circ\xi_n([N_n,\infty))$ and $f\circ\xi([N_n,\infty))$ lie in the same path component of $T_X\backslash B([x_0],r)$ for large enough $n$.
	
	Clearly, $\xi_n([N_n,\infty))$ and $\xi([N_n,\infty))$ lie in  $X\backslash B(x_0,r)$ for large enough $n$, and by \Cref{PathCompCor}, they must lie in the same path component. Hence $\text{end}(\xi_n)\to\text{end}(\xi)$, so as $U$ is closed, we must have $\text{end}(\xi)\in U$, so $\text{end}(f\circ\xi)\in F(U)$, so $F(U)$ is closed. Therefore $F^{-1}$ is continuous.
\end{proof}

We can therefore obtain the following result.

\begin{prop}
	\label{LocallyFiniteEnds}
	Let $(X,d)$ be a proper geodesic metric space. Then there exists a locally finite simplicial tree $T$ such that $\emph{Ends}(X)$ is homeomorphic to $\emph{Ends}(T)$.
\end{prop}

\begin{proof}
	As the space of ends is preserved by quasi-isometry, it suffices to show that $T_X$ is quasi-isometric to a locally finite simplicial tree. By \Cref{ProperFinite}, it therefore suffices to show that $T_X$ is proper.
	
	Let $r>0$, and recall that for the end-approximating map $f:X\to T_X$, we have that $f^{-1}(B([x_0],r))=B(x_0,r)$. Let $([x_n])$ be a sequence in $B([x_0],r)$, and consider the sequence of representatives $(x_n)$ in $B(x_0,r)$. As $X$ is proper, $B(x_0,r)$ is compact, so there exist some subsequence $(x_{n_k})$ and some point $x\in B(x_0,r)$ such that $x_{n_k}\to x$. As $f$ is continuous, we have that $[x_{n_k}]\to [x]$. We know that $[x]\in B([x_0],r)$, so $B([x_0],r)$ is compact.
	
	Let $[y]\in T_X$, and let $R>0$. There exists $r>0$ such that $B([y],R)\subset B([x_0],r)$. Therefore $B([y],R)$ is a closed subset of a compact set, so is itself compact. We conclude that $T_X$ is a proper metric space.
\end{proof}

\subsection{Visual metrics on the boundaries of quasi-trees}

For a tree, its space of ends is homeomorphic to its boundary. This is consequently also true for quasi-trees, as both the space of ends and the geodesic boundary are preserved up to homeomorphism by quasi-isometries. In the case where $X$ is a proper quasi-tree, the result of \Cref{EndsHomeomorphic} is a trivial consequence of this fact.

Not only is the topology on the boundaries preserved by quasi-isometries, some of the metric structure is too. Given a quasi-isometry between two hyperbolic geodesic metric spaces, the induced homeomorphism between their boundaries is known to be quasi-symmetric \cite{Bonk2000,Bridson1999,Mackay2010}. There is a family of metrics on these boundaries such that this is the case, known as the visual metrics.

When we have a $(1,C)$-quasi-isometry between spaces, it is known that the induced homeomorphism between the boundaries will in fact be a bi-Lipschitz map with respect to the visual metrics \cite{Bonk2000}. In particular, by how visual metrics are defined, this means that the pullback of a visual metric on one of the boundaries will be a visual metric on the other, so the homeomorphism can be instead viewed as an isometry.

In the case of $\mathbb{R}$-trees, given a choice of parameter, the visual metrics on the boundary with that parameter are all bi-Lipschitz equivalent to a visual metric which can be written down explicitly in a standard way. By \Cref{QuasiIsometry} and the above, we can choose a metric on the boundary of a quasi-tree $X$ such that it is isometric to the boundary of $T_X$ under this standard metric. Moreover, this metric on the boundary of $X$ is a natural extension of the standard visual metric on the boundary of $T_X$.

Many of the definitions and facts about boundaries used here are taken from the paper on the topic by Kapovich and Benakli \cite{Kapovich2002}.

\begin{con}
	When working with metrics on boundaries of $X$, everything will be done with respect to a basepoint $x_0\in X$. As in previous sections, this will be assumed throughout to have already been chosen. The choice of basepoint does not matter (see \Cref{Basepoint}).
\end{con}

\begin{defn}
	Let $(X,d)$ be a hyperbolic metric space. A sequence $(x_n)$ in $X$ \emph{converges to infinity} if $\lim_{i,j\to\infty}(x_i,x_j)_{x_0}=\infty$. We say that two such sequences are \emph{equivalent}, and write $(x_n)\sim(y_n)$, if $\liminf_{i,j\to\infty}(x_i,y_j)_{x_0}=\infty$. The \emph{sequential boundary} of $X$ is the set of equivalence classes of sequences that converge to infinity, and is denoted by $\partial X$.
\end{defn}

\begin{note}
	When $p\in\partial X$, and a sequence $(x_n)$ converging to infinity in X belongs to the equivalence class $p$, we write $x_n\to p$.
\end{note}

As mentioned before, it is possible to put metrics on this boundary, although there is no single canonical metric. To define the family of metrics we are interested in, we first note that the Gromov product can be extended to the sequential boundary.

\begin{defn}
	Let $X$ be a hyperbolic metric space. For $p,q\in\partial X$ the \emph{Gromov product} of $p$ and $q$ is
	\begin{equation*}
	(p,q)_{x_0}=\sup\{\liminf_{i,j\to\infty}(x_i,y_j)_{x_0}:x_i\to\xi, y_j\to\gamma\}.
	\end{equation*}
\end{defn}

\begin{defn}
	Let $X$ be a hyperbolic metric space, and let $a>1$. A metric $d_{\partial X}$ on $\partial X$ is a \emph{visual metric} with respect to the \emph{visual parameter} $a$ if there exists $C\geqslant 1$ such that for every $p,q\in\partial X$
	\begin{equation*}
	\frac{1}{C}a^{-(p,q)_{x_0}}\leqslant d_{\partial X}(p,q)\leqslant Ca^{-(p,q)_{x_0}}.
	\end{equation*}
\end{defn}

It is known that for any hyperbolic space $X$, there exists $a_0>1$ such that for every $a\in(1,a_0)$, the sequential boundary $\partial X$ admits a visual metric with parameter $a$, see \cite[Chapter III.H]{Bridson1999} or \cite{Groves2019} for example. If $X$ is an $\mathbb{R}$-tree, then $a_0=\infty$, and for every parameter $a>1$, there is a standard visual metric given by $d_{\partial X}(p,q)= a^{-(p,q)_{x_0}}$.

There is another notion of boundary that applies when $X$ is also a geodesic space.

\begin{defn}
	Let $(X,d)$ be a metric space, and let $A$ and $B$ be subsets of $X$. The \emph{Hausdorff distance} between $A$ and $B$ is
	\begin{equation*}
	d_H(A,B)=\inf\{\varepsilon\geqslant 0 : A\subset N(B,\varepsilon)\text{ and }B\subset N(A,\varepsilon)\},
	\end{equation*}
	where $N(A,\varepsilon)$ is the closed $\varepsilon$-neighbourhood of $A$.
\end{defn}

\begin{defn}
	Let $(X,d)$ be a proper, hyperbolic, and geodesic metric space. We say that two geodesic rays $p:[0,\infty)\to X$ and $q:[0,\infty)\to X$ are equivalent, and write $p\sim q$, if $d_H(p,q)$ is finite. The \emph{geodesic boundary} of $X$ is the set of equivalence classes of geodesic rays in $X$, and is denoted by $\partial_{\infty} X$.
\end{defn}

It is again possible to put a metric on this boundary, although as before there is no single canonical metric.

\begin{rem}
	For any such space $X$, and any $p,q\in \partial_{\infty}X$ that are not equal, there will exist a bi-infinite geodesic $\gamma:\mathbb{R}\to X$ from $p$ to $q$, in the sense that $\gamma|_{[0,\infty)}\in q$ and $\widehat{\gamma}|_{[0,\infty)}\in p$, where $\widehat{\gamma}(t)=\gamma(-t)$. When $X$ is an $\mathbb{R}$-tree, this bi-infinite geodesic is unique.
\end{rem}

\begin{defn}
	Let $(X,d)$ be a proper, hyperbolic, and geodesic metric space, and let $a>1$. A metric $d_{\partial_{\infty} X}$ on $\partial_{\infty} X$ is a \emph{visual metric} with respect to the \emph{visual parameter} $a$ if there exists $C\geqslant 1$ such that for every $p,q\in\partial_{\infty} X$, and every bi-infinite geodesic $\gamma$ from $p$ to $q$, we have that
	\begin{equation*}
	\frac{1}{C}a^{-d(x_0,\gamma)}\leqslant d_{\partial_{\infty} X}(p,q)\leqslant Ca^{-d(x_0,\gamma)}.
	\end{equation*}
\end{defn}

As with the sequential boundary, for any such space $X$, there exists $a_0>1$ such that for every $a\in(1,a_0)$, the geodesic boundary $\partial_{\infty} X$ admits a visual metric with parameter $a$. If $X$ is an $\mathbb{R}$-tree, then $a_0=\infty$, and for every parameter $a>1$, there is a standard visual metric given by $d_{\partial_{\infty} X}(p,q)= a^{-d(x_0,\gamma)}$ \cite{Kapovich2002}.

\begin{rem}
	\label{Basepoint}
	The visual metrics on the sequential and geodesic boundaries are defined in terms of a basepoint, however on either boundary two visual metrics with the same parameter but different basepoints will be bi-Lipschitz equivalent \cite{Kapovich2002,Buyalo2007}.
\end{rem}

It is well known that a quasi-isometry $f$ between two hyperbolic spaces $X$ and $Y$ induces a canonical homeomorphism $\partial f:\partial X\to \partial Y$, and if $X$ and $Y$ are also geodesic, then the same is true for $\partial_{\infty} f:\partial_{\infty}X\to\partial_{\infty}Y$. This homeomorphism has strong metric properties depending on the strength of the quasi-isometry, see \cite{Buyalo2007,Kapovich2002}, however we will be interested in the following case.

\begin{thm}
	\emph{\cite[Theorem~6.5]{Bonk2000}}
	Let $X$ and $Y$ be hyperbolic geodesic metric spaces, and let $f:X\to Y$ be a $(1,C)$-quasi-isometry. Let $d_{\partial X}$ and $d_{\partial Y}$ be visual metrics with parameter $a>1$ on $\partial X$ and $\partial Y$, respectively. Then the induced map $\partial f:(\partial X,d_{\partial X})\to(\partial Y,d_{\partial Y})$ is bi-Lipschitz.
\end{thm}

This may not be immediately obvious from reading Theorem 6.5 in \cite{Bonk2000}, however we note that what they call a $(1,C)$-rough similarity is exactly a $(1,C)$-quasi-isometry, and moreover that if $f:X\to Y$ is a $(1,C)$-quasi-isometry then there exists $K\geqslant 0$ such that for every $x,y,x_0\in X$ we have that 
\begin{equation*}
	(x,y)_{x_0}-K\leqslant(f(x),f(y))_{f(x_0)}\leqslant (x,y)_{x_0}+K.
\end{equation*}
Their reasoning therefore gives us that $\partial f$ is what they call a $(1,\lambda)$-snowflake map, which is precisely a bi-Lipschitz map.

\begin{rem}
	As noted before, the family of visual metrics on a boundary with the same parameter is a bi-Lipschitz equivalence class, so if we define a metric $d_X$ on $\partial X$ by $d_{\partial X}(p,q)=d_{\partial Y}(\partial f(p),\partial f(q))$ then this would be a visual metric on $\partial X$, which would allow us to view $\partial f$ as an isometry.
\end{rem}

In a hyperbolic geodesic space $X$, it is well known that there is a natural homeomorphism between $\partial X$ and $\partial_{\infty} X$, and moreover that visual metrics on $\partial X$ and $\partial_{\infty} X$ with the same parameter are bi-Lipschitz equivalent, for example, see \cite{Drutu2018,Hamann2018}. We therefore also get the following from \cite{Bonk2000}.

\begin{cor}
	\emph{\cite{Bonk2000}}
	Let $X$ and $Y$ be hyperbolic geodesic metric spaces, and let $f:X\to Y$ be a $(1,C)$-quasi-isometry. Let $d_{\partial_{\infty} X}$ and $d_{\partial_{\infty} Y}$ be visual metrics with parameter $a>1$ on $\partial_{\infty} X$ and $\partial_{\infty} Y$, respectively. Then the induced map $\partial_{\infty} f:(\partial_{\infty} X,d_{\partial_{\infty} X})\to(\partial_{\infty} Y,d_{\partial_{\infty} Y})$ is bi-Lipschitz.
\end{cor}

\begin{cor}
	\label{SequentialRealv2}
	Let $X$ be a quasi-tree. For every $a>1$, there exists a visual metric $d_{\partial X}$ on $\partial X$ with visual parameter $a$ such that $(\partial X,d_{\partial X})$ is isometric to $(\partial T_X,d_{\partial T_X})$, where $d_{\partial T_X}$ is the standard visual metric on $\partial T_X$ with visual parameter $a$.
\end{cor}

\begin{proof}
	Let $f:X\to T_X$ be the end-approximating map. Then \Cref{QuasiIsometry} tells us that $f$ is a $(1,C)$-quasi-isometry. Let $\partial f:\partial X\to T_X$ be the induced bi-Lipschitz homeomorphism between the boundaries, and for $a>1$ let $d_{\partial T_X}$ be the standard visual metric with parameter $a$ on $\partial T_X$. Then $d_{\partial X}(p,q)=d_{\partial T_X}(\partial f(p),\partial f(q))$ is a visual metric on $\partial X$ with visual parameter $a$, and $\partial f$ acts as an isometry between the boundaries.
\end{proof}

The exact same statement holds if we replace $T_X$ by $\Gamma_X$, the simplicial tree constructed from $T_X$ in Section 4.2. We also get the same result for both $T_X$ and $\Gamma_X$ if we suppose that $X$ is a proper quasi-tree, and then consider the geodesic boundaries.

In the case where the metric on $\partial X$ is the pullback of the standard visual metric on $\partial T_X$ (with respect to some parameter and basepoint), we can in fact define this metric explicitly in a very natural way, using only the geometry of $X$. We start with the sequential boundary. Let $p,q\in\partial X$. Then for any $a>1$ we define
\begin{equation*}
	d_{\partial X}(p,q)=d_{\partial T_X}(\partial f(p),\partial f(q))=a^{-(\partial f(p),\partial f(p))^*_{[x_0]}}.
\end{equation*}
Let $x_i\to p$ and $y_j\to q$. The homeomorphism $\partial f: X\to T_X$ is defined such that $[x_n]\to \partial f(p)$ and $[y_n]\to \partial f(q)$. From the definition of the Gromov product on the boundary, the geometry of an $\mathbb{R}$-tree, and \Cref{PathCompCor} regarding path components we can see that
\begin{align*}
(\partial f(p),\partial f(q))^*_{[x_0]} & =\sup\{r\in[0,\infty):\text{ for }[x_n]\to \partial f(p)\text{ and }[y_n]\to \partial f(q)\text{ there are infinitely many}
\\ & \quad\text{terms in each sequence that lie in the same path component of } T_X\backslash B([x_0],r)\}
\\ & =\sup\{r\in[0,\infty):\text{ for }x_n\to p\text{ and }y_n\to q\text{ there are infinitely many terms}
\\ & \quad\text{in each sequence that lie in the same path component of }X\backslash B(x_0,r)\}.
\end{align*}
Note that the metric on $\partial X$ is given in terms of this value. In particular, it is clear that if $X$ is an $\mathbb{R}$-tree, then this simply returns the standard visual metric with parameter $a$, as expected given that $f$ is the identity in this case.

We can do the same thing with the geodesic boundary. Let $p,q\in\partial_{\infty}X$. As $f$ is an isometry when restricted to geodesics based at $x_0$, and $T_X$ is an $\mathbb{R}$-tree, the image of the equivalence class $p$ under $\partial_{\infty} f$ will be a single geodesic in $T_X$, with the same being true for the image of $q$. For any $a>1$, we define
\begin{equation*}
d_{\partial_{\infty} X}(p,q)=d_{\partial_{\infty} T_X}(\partial_{\infty} f(p),\partial_{\infty} f(q))=a^{-d^*([x_0],[z])},
\end{equation*}
where $[[x_0],[z]]=p'\cap q'$, with the geodesics $p'$ and $q'$ being the images of the equivalence classes $p$ and $q$. In particular, we can see that
\begin{align*}
d^*([x_0],[z]) & =\inf\{r\in[0,\infty):p'\backslash B([x_0],r)\text{ and }q'\backslash B([x_0],r)\text{ lie in different  path}
\\ & \quad\text{ components of }T_X\backslash B([x_0],r)\}
\\ & =\inf\{r\in[0,\infty):r_p\backslash B(x_0,r)\text{ and }r_q\backslash B(x_0,r)\text{ lie in different path}
\\ & \quad\text{ components of }X\backslash B(x_0,r),\text{ for any choice of }r_p\in p\text{ and }r_q\in q\},
\end{align*}
and we note that $d_{\partial_{\infty}X}(p,q)$ is given in terms of this value. Again, if $X$ is an $\mathbb{R}$-tree then this will give us the standard visual metric with parameter $a$, so these metrics on the boundaries of quasi-trees can be viewed as extensions of the standard visual metrics on the boundaries of $\mathbb{R}$-trees.

\bibliographystyle{alpha}
\bibliography{/home/kerr/Nextcloud/Documents/Projects/references}

\end{document}